\documentclass{cmslatex}
\usepackage[paperwidth=7in, paperheight=10in, margin=.875in]{geometry}
\usepackage[backref,colorlinks,linkcolor=red,anchorcolor=green,citecolor=blue]{hyperref}
\usepackage{amsfonts,amssymb,amsmath,graphicx}
\usepackage{inputenc}
\usepackage{footnote}
\usepackage{float}
\usepackage{enumerate}
\usepackage{cite}
\usepackage{ulem,booktabs,multirow}
\usepackage{bm,units}
\usepackage{physics}
\usepackage{tikz}
\usepackage{verbatim}
\usepackage{subfigure}
\usepackage{xr}
\usepackage{lineno}

\usepackage{xcolor}

\renewcommand{\d}{\mathop{}\!\mathrm{d}} 
\newcommand{\D}{\mathrm{D}} 

\newcommand{\pt}{\partial}

\newcommand{\ve}{\varepsilon}

\newcommand{\e}{\mathrm{e}}  
\newcommand{\ii}{\mathrm{i}}  

\newcommand{\BbbC}{\mathbb{C}}
\newcommand{\BbbE}{\mathbb{E}}
\newcommand{\BbbN}{\mathbb{N}}
\newcommand{\BbbR}{\mathbb{R}}
\newcommand{\BbbS}{\mathbb{S}}

\newcommand{\dist}{\operatorname{dist}}

\renewcommand{\emph}{\textit}

\newcommand{\fracd}[2]{\frac{\displaystyle #1}{\displaystyle #2}}

\makeatletter 
\@ifundefined{tr}{
    \newcommand{\tr}{\operatorname{tr}}
}
\@ifundefined{divergence}{
    
}{
    
}
\makeatother

\allowdisplaybreaks

\definecolor{LightCerulean}{RGB}{118,210,251}
\definecolor{Raspberry}{RGB}{247, 67, 140}
\definecolor{PaleGold}{RGB}{252, 239, 164}
\definecolor{PaleRaspberry}{RGB}{250, 179, 209}
\definecolor{tabblue}{RGB}{ 31, 119, 180 }
\definecolor{taborange}{RGB}{ 255, 127, 14 }
\definecolor{tabgreen}{RGB}{ 44, 160, 44 }
\definecolor{tabred}{RGB}{ 214, 39, 40 }
\definecolor{tabpurple}{RGB}{ 148, 103, 189 }
\definecolor{tabbrown}{RGB}{ 140, 86, 75 }
\definecolor{tabpink}{RGB}{ 227, 119, 194 }
\definecolor{tabgray}{RGB}{ 127, 127, 127 }
\definecolor{tabolive}{RGB}{ 188, 189, 34 }
\definecolor{tabcyan}{RGB}{ 23, 190, 207 }

\usepackage{tikz-cd}
\usepackage{stmaryrd}
\usepackage{lineno}
\usetikzlibrary{decorations.markings}

\title{Uniform asymptotics and entropy structure of the Bingham distribution in arbitrary dimensions
\thanks{This work is funded by the National Natural Science Foundation of China (Nos.~12225102, T2321001, 12288101).}}
\author{Dawei Wu \thanks{School of Mathematical Sciences, Peking University, Beijing 100871, China (2201110052@pku.edu.cn).}
\and Lei Zhang \thanks{School of Mathematical Sciences, Beijing International Center for Mathematical Research, Center for Quantitative Biology, Center for Machine Learning Research, Peking University, Beijing 100871, China (zhangl@math.pku.edu.cn)}
\and Pingwen Zhang \thanks{School of Mathematics and Statistics, Wuhan University, Wuhan 430072, China; School of Mathematical Sciences, Peking University, Beijing 100871, China (pzhang@pku.edu.cn)}}

\begin{document}
\maketitle

\begin{abstract}
The Bingham distribution is widely used to model directional data with antipodal symmetry, and its normalizing constant $Z$ and moments play a central role in statistical inference and closure models.
Their behavior becomes singular when one or more eigenvalue gaps of the parameter matrix become unbounded.
In this work, we develop a uniform asymptotic framework for the Bingham distribution in arbitrary dimensions. 
Starting from an inverse-Laplace integral representation, we derive asymptotic expansions with explicit remainder estimates that are uniform with respect to arbitrary relative scales among multiple diverging eigenvalues. In a particular regime, the asymptotic series becomes absolutely convergent with an exponentially small remainder.
These results yield precise asymptotic profiles of the Bingham moments and characterize the degeneration of higher-dimensional distributions to lower-dimensional counterparts.
As an application to the Bingham closure, we prove that the entropy admits a decomposition into an explicit logarithmic leading term and a residual that is uniformly Lipschitz continuous on the moment simplex. Moreover, on each boundary face, the residual agrees with its lower-dimensional version up to an explicit additive constant. 
The results provide a unified description of the singular structure of the Bingham distribution and have implications for numerical closure and $Q$-tensor models of nematic liquid crystals.
\end{abstract}



\section{Introduction}

The Bingham distribution is a well-known probability distribution on the unit sphere, originally proposed by C.~Bingham in three dimensions \cite{bingham_distribution_1964,bingham_antipodally_1974}. In this paper, we study the generalized case in arbitrary dimension $d$, which is defined as follows.
\begin{definition}
The \textbf{$d$-dimensional Bingham distribution} $\mathrm{Bing}(B)$ has the probability density function
\begin{equation} \label{bingham}
    f(m)=\frac{1}{\omega_d Z} \exp(m\cdot Bm),
\end{equation}
defined for $m$ on the unit sphere $\BbbS^{d-1}\subset\BbbR^d$ with respect to the surface measure $\d m$. Here, the normalizing constant equals
\begin{equation} \label{bingham-Z}
    Z=\frac{1}{\omega_d}\int_{\BbbS^{d-1}} \exp(m\cdot Bm)\d m,
\end{equation}
$\omega_d=\frac{2\pi^{\frac{d}{2}}}{\Gamma(\frac{d}{2})}$ is the area of $\BbbS^{d-1}$, and $B$ is a $d\times d$ symmetric matrix.
\end{definition}
\begin{remark}
Conventionally, the parameter $B$ is restricted to the traceless matrices, but as we will explain later, this is merely one of many gauge conditions (i.e.~a choice of representative within the equivalence class $B\sim B+tI$). Therefore, we define the Bingham distribution here with a general symmetric $B$, and will discuss the gauge condition in more detail in Subsection \ref{subsec:gauge}.
\end{remark}
The Bingham distribution is the simplest distribution on $\BbbS^{d-1}$ with antipodal symmetry, i.e.~ensuring equivalence between $m$ and $-m$, and has become very useful in directional statistics \cite{mardia_statistics_1975}. 
Applications of the Bingham distribution include the modeling of demagnetization paths in palaeomagnetism \cite{kirschvink_least-squares_1980}, head-to-tail symmetric nematic liquid crystals \cite{ball_nematic_2010}, and high-dimensional parameters in variational auto-encoders \cite{davidson_hyperspherical_2018}.

Many analytic and computational questions concerning the Bingham distribution reduce to the evaluation of the normalizing constant $Z$ \eqref{bingham-Z} and its derivatives. Taking the matrix gradient (denoted by $\nabla$) with respect to $\ln Z(B)$, one gets the second-order moments of $\mathrm{Bing}(B)$:
\begin{equation} \label{bingham-mmnt}
\nabla\ln Z(B) =\frac{\nabla Z(B)}{Z(B)}=
\fracd{\int_{\BbbS^{d-1}} mm^T \exp(m\cdot Bm)\d m}{\int_{\BbbS^{d-1}}\exp(m\cdot Bm)\d m} = \mathbb{E}(mm^T).
\end{equation}
The moments $\BbbE(mm^T)$ obtained above are crucial in statistical modeling of directional data as they are directly related to observed data. The inverse determination of a Bingham distribution from $\BbbE(mm^T)$ is also an important problem known as the \textbf{Bingham closure}.
One example is the maximum-likelihood estimate \cite{chen_maximum_2021,kume_exact_2018}. With independent data points collected as $\{m_1,\ldots,m_N\}$, the negative log-likelihood according to \eqref{bingham} can be expressed by
\[-\log\mathcal{L}(B)=-\frac1N \ln\prod_{j=1}^N f(m_j;B) = \ln(\omega_d Z)-\qty(\frac1N\sum_{j=1}^N m_j m_j^T):B.\]
Taking the gradient with respect to $B$, we find that the maximum-likelihood estimator is exactly the Bingham closure from the empirical second-order moment.
Another classical example is the $Q$-tensor theory of nematic liquid crystals \cite{ball_nematic_2010,mottram_introduction_2014}, where the moment tensor $\BbbE(mm^T)$ in three dimensions (3D) is extensively used as the principal order parameter. In order to be physically correct, the free energy functional requires the underlying Bingham distribution as a quasi-equilibrium approximation \cite{liu_axial_2005}, and is computed by Bingham closure. Here, the (negative) entropy function $S$ is involved in the physical free energy as a function of $\BbbE(mm^T)$:
\begin{equation} \label{bingham-ent}
    S = \int_{\BbbS^{d-1}} f\ln (\omega_d f)\d m = B:\BbbE(mm^T) - \ln Z.
\end{equation}

Analysis and computation of the Bingham distribution reduce largely to evaluating $Z(B)$ and its derivatives.
As observed by Bingham \cite{bingham_distribution_1964,bingham_antipodally_1974}, $Z(B)$ is a confluent hypergeometric function of matrix argument and has no elementary expression. 
The main difficulty arises in singular regimes where one or more eigenvalue gaps of $B$ become unbounded. In such regimes, some moments vanish while the entropy \eqref{bingham-ent} diverges logarithmically \cite{ball_nematic_2010}, which makes both analysis and stable numerical evaluation substantially more delicate.
Existing methods, including Taylor expansions, interpolation \cite{luo_fast_2018}, and holonomic gradient methods \cite{kume_exact_2018,sei_calculating_2015} are only designed for bounded or moderately-sized $B$.
The saddlepoint approximation method \cite{kume_saddlepoint_2005} and some asymptotic methods on the dimensionality $d$ \cite{bagyan_complete_2024} have asymptotic exactness when $d$ goes to infinity, but do not provide explicit error estimates when several eigenvalues diverge at different rates.
In three dimensions, asymptotic formulas for $Z$ and its derivatives were obtained by Bingham \cite{bingham_distribution_1964} and Kent \cite{kent_asymptotic_1987}, and the Bingham closure was studied in detail by the liquid crystal modeling community \cite{ball_nematic_2010,wang_modelling_2021}. More recently, the authors derived a decomposition of the 3D Bingham entropy into a singular leading term and a Lipschitz correction term \cite{shi_molecular_2026}.
What is still missing is a general-dimensional theory that treats simultaneously multiple unbounded eigenvalues without imposing assumptions on their relative sizes, and that connects the resulting asymptotic structure of $Z$ and its moments to the singular geometry of the Bingham entropy.

In this paper, we develop such a uniform asymptotic theory for the Bingham distribution in arbitrary dimensions. Our contributions are three-fold.
First, we start from an inverse-Laplace integral representation \cite{sei_calculating_2015} to derive exponentially convergent integral formulas for $Z$ and all its derivatives (Theorem \ref{thm:DZ-keyhole-int}).
These representations allow us to obtain their asymptotic expansions when an arbitrary subset of eigenvalues tends to $-\infty$, with explicit remainder estimates that remain uniform with respect to the remaining bounded parameters, and, in particular, do not require any prescribed relative scaling among the diverging eigenvalues (Theorem \ref{thm:Mdk-asymp-mult} and Corollary \ref{cor:DMdk-asymp-mult}). The coefficients of the expansions are naturally expressed in terms of lower-dimensional Bingham quantities, revealing an intrinsic dimensional hierarchy.
Moreover, when the number of the bounded parameters is even, the formal asymptotic expansion becomes an absolutely convergent infinite series, with an exponentially small remainder (Theorem \ref{thm:uni-conv-even}).
Second, the uniform expansions yield precise asymptotic profiles for even-ordered Bingham moments (Proposition \ref{prop:za-asymp}). In particular, the second-order moments $z_j=\BbbE(m_j^2)$ associated with a diverging eigenvalue $\mu_j$ decay at the explicit rate $z_j\sim (2|\mu_j|)^{-1}$, whereas the remaining moments converge to those of the corresponding lower-dimensional Bingham distribution. 
Third, using the moment asymptotics, we characterize the singular structure of the Bingham entropy $S$ under Bingham closure,  and prove the decomposition
\[ S=-\frac12\ln(z_1\cdots z_d)+\Delta S, \]
where $\Delta S$ is uniformly Lipschitz continuous and extends to the boundary of the moment simplex (Theorem \ref{thm:Sq-asymp}). 
We further show that the boundary values of $\Delta S$ inherit the same dimensional hierarchy: on every face of the moment simplex, $\Delta S$ coincides with the lower-dimensional entropy residual up to an explicit additive constant (Theorem \ref{thm:dS-embed}).

\section{Integral representation}

In this section, we briefly present the integral representation of $Z$ and its derivatives \cite{chen_maximum_2021,sei_calculating_2015}. 

\subsection{Diagonalization and gauge condition} \label{subsec:gauge}

First, we need to diagonalize $B$ in \eqref{bingham} with a rotation, and normalize it by removing one trivial degree of freedom, which greatly simplifies the parameter.

Using a coordinate frame $R\in\mathrm{SO}(d)$ composed of the eigenvectors of $B$ (so that $R^TBR=\diag(\mu_1,\ldots,\mu_d)$ is diagonal), and changing the coordinate $\widetilde m = R^T m$, we find that the density function with respect to $\widetilde m$ equals
\[\tilde f(\widetilde m) =f(R\widetilde m)= \frac{1}{\omega_d Z(B)} \exp(\widetilde m \cdot R^TBR \widetilde m)
=\frac{1}{\omega_d Z(B)} \exp(\sum_{j=1}^d \mu_j \widetilde m_j^2).\]
That is, the distribution $m\sim\mathrm{Bing}(B)$ can be rotated to another distribution $\widetilde m\sim\mathrm{Bing}(\diag(\mu))$ with diagonal parameters.
Therefore, we can assume w.l.o.g.~that $B=\diag(\mu)=\diag(\mu_1,\ldots,\mu_d)$. By a harmless abuse of notation, we also denote the distribution by $\mathrm{Bing}(\mu)$, and write $Z$ as a function of $\mu$:
\begin{equation} \label{Z-mu}
    Z(\mu)=\frac{1}{\omega_d} \int_{\BbbS^{d-1}} \exp\left( \sum_{j=1}^d \mu_j m_j^2 \right)\d m.
\end{equation}
Evidently, $Z(\mu)$ is positive, smooth and symmetric with respect to its arguments. 

Since $|m|=1$, shifting all $\mu_j$'s by the same amount does not change the distribution:
\begin{equation}\label{Z-mupt-eq}
    \mathrm{Bing}(\mu) = \mathrm{Bing}(\mu+t),\ Z(\mu+t) = \e^t Z(\mu),\ \forall t\in\BbbR,
\end{equation} 
where vector-scalar addition is defined by $\mu+t\triangleq(\mu_1+t,\cdots,\mu_d+t)$. 
One derives \eqref{Z-mupt-eq} straightforwardly by noticing that
\[ f(m; \mu+t) = \frac{1}{\omega_d Z(\mu+t)} \exp(\sum_{j=1}^d \mu_j m_j^2 + t|m|^2) \overset{|m|^2=1}{=\!=\!=} \frac{\e^t Z(\mu)}{Z(\mu+t)} f(m;\mu),\]
and that both densities normalize to 1.
Conversely, $\mathrm{Bing}(\mu)=\mathrm{Bing}(\mu')$ implies that $\mu-\mu'$ is made of constants, as is seen by evaluating the densities at the coordinate
vectors $e_j$.

Thus, Bingham parameters are defined modulo common shifts, and a specific gauge needs to be chosen to make the parameter unique.
The conventional gauge is $\sum_j\mu_j=0$. For asymptotic analysis, however, we use the ordered gauge
\[0=\mu_1\ge \mu_2\ge\cdots\ge\mu_d,\]
which will be stated whenever needed.

\subsection{Hankel-type integral}

From now on, we will focus on the case where $B$ is diagonal, and study the normalizing constant $Z(\mu)$ \eqref{Z-mu}. 
To normalize the argument $\mu$, we will assume that $0=\mu_1\ge \mu_2\ge\cdots \ge \mu_d$, which will be explicitly stated at every usage for disambiguation.
We provide the integral representation of $Z(\mu)$ as an inverse Laplace transform. This result has been derived in \cite{kume_exact_2018,sei_calculating_2015}, but for completeness we present a proof in Appendix \ref{app:proof-of-lt}.
\begin{theorem} \label{thm:Z-lt}
For any $\mu\in\BbbR^d$, let $\lambda>\max_j \mu_j$, and then
\begin{equation}\label{Zd-is-inv-lt}
    Z(\mu) = \frac{\Gamma(\frac{d}{2})}{2 \pi\ii} \int_{\lambda-\ii\infty}^{\lambda+\ii\infty} \e^z \prod_{j=1}^d (z-\mu_j)^{-\frac12} \d z.
\end{equation}
Here and in what follows, the complex powers are defined by
\[z^s\triangleq \exp(s\ln z),\]
where $\ln z$ is defined such that $\Im(\ln z)=\arg z\in(-\pi,\pi)$. In particular, this choice ensures that $\Re z^{\frac12}>0$ for all $z\notin \BbbR_-.$ 
\end{theorem}

Denoting by $\alpha=(\alpha_1,\ldots,\alpha_d)\in\BbbN^d$ a multi-index, we can differentiate \eqref{Z-mu} over $\mu$ to get
\begin{equation} \label{DZd-mu}
    \D^\alpha Z(\mu) = \frac{1}{\omega_d} \int_{\BbbS^{d-1}} m_1^{2\alpha_1}\cdots m_d^{2\alpha_d} \e^{\sum_j \mu_j m_j^2}\d m,
\end{equation}
where $|\alpha|=\alpha_1+\cdots+\alpha_d$, $\D^\alpha = \frac{\pt^{|\alpha|}}{\pt \mu_1^{\alpha_1}\cdots \pt \mu_d^{\alpha_d}}$.
That is, the derivatives represent the (not yet normalized) even-order moments of $\mathrm{Bing}(\mu)$.
Similar to \eqref{Zd-is-inv-lt}, we can express them as inverse Laplace transforms.
\begin{corollary}
Work under the same assumptions as Theorem \ref{thm:Z-lt}. 
For any multi-index $\alpha\in\BbbN^d$, we have that
\begin{equation} \label{DZd-inv-lt}
    \D^\alpha Z(\mu) = \frac{\Gamma(\frac{d}{2})}{2\pi\ii} \prod_{j=1}^d \frac{(2\alpha_j-1)!!}{2^{\alpha_j}} \int_{\lambda-\ii\infty}^{\lambda+\ii\infty} \e^z \prod_{j=1}^d (z-\mu_j)^{-\frac12-\alpha_j} \d z,
\end{equation}
where $0!!=(-1)!!=1$ by convention.
\end{corollary}

Formally, \eqref{DZd-inv-lt} takes the derivatives with respect to $\mu$ under the integration in \eqref{Zd-is-inv-lt}, but its validity requires the convergence of the original integral. In fact, the convergence rate of the integrals in \eqref{Zd-is-inv-lt} and \eqref{DZd-inv-lt} are only of polynomial order, and \eqref{Zd-is-inv-lt} is not absolutely convergent if $d\le 2$.

To justify the integration formulas above and their differentiation, we need to deform the contour of integration.
We introduce the following well-known contour wrapped around the negative real axis, known as Hankel's loop \cite{abramowitz_handbook_2013}.

\begin{definition}
A piecewise $C^1$ curve $C: (-\infty,\infty)\to \BbbC$ is called a \textbf{(general) Hankel-type contour} if:
\begin{enumerate}[(a)]
\item There exists a polynomial $P$ such that $|\Im C(t)| \le P(\Re C(t))$ for all $t\in\BbbR$;
\item There exists $M>0$ such that $\Im C(t)>0$ when $t>M$, and $\Im C(t)<0$ when $t<-M$;
\item $\Re C(t) \to -\infty$ as $t\to\pm\infty$;
\item $\dist(C(t), \BbbR_-)>0, \forall t$.
\end{enumerate}
Geometrically, a Hankel-type contour runs from $-\infty$ below the negative real axis, encircles the origin counterclockwise, and returns to $-\infty$ above the axis; its imaginary part grows at most polynomially.
\end{definition}

Integration along the Hankel-type contours has the following important property. The proof requires only elementary complex analysis.

\begin{proposition} \label{prop:hankel-int}
Let $g:\BbbC\setminus (-\infty,0]\to\BbbC$ be analytic.
If $g(z)$ has a polynomial growth, i.e.~$|g(z)| \le L|z|^p$ for some constant $p\ge 0$ as $z\to\infty$, then the integral of $G(z)=\e^z g(z)$ over any Hankel-type contour is the same.
\end{proposition}
\begin{proof}[Sketch of proof]
We compare an arbitrary contour $C$ with a standard keyhole contour $C_0$ by connecting their truncated endpoints with vertical segments (see Figure \ref{fig:any-hankel}). By Cauchy's theorem, integration of $G(z)$ over any closed loop is zero since the slit plane is simply connected, so the difference of the two truncated integrals equals the integrals over these segments. 
Since $|g(z)|\leq L|z|^p$ and the segment lengths grow at most polynomially in $\Re z$, each connecting integral is bounded by a polynomial factor times $e^{\Re z}$, and tends to zero as the endpoints tend
to $-\infty$. Therefore, the integrals over $C$ and $C_0$ coincide.


\begin{figure}
\centering
\begin{tikzpicture}[decoration={markings, 
    mark=at position 0.6 with {\arrow{stealth}}}, line cap=round]
\draw[gray,->](-3,0)--(3,0);
\draw[line width=2pt,opacity=0.5](-3,0)--(0,0);
\draw[gray,->](0,-3)--(0,3);
\fill(0,0) circle[radius=2pt];
\draw[red,thick,postaction={decorate}] (-3,-.2)--(0,-.2);
\draw[red,thick,postaction={decorate}] (0,-.2) arc[radius=.2,start angle=-90,end angle=90];
\draw[red,thick,postaction={decorate}] (0,.2)--(-3,.2) node[anchor=south east]{$C_0$};
\draw[thick,postaction={decorate}] (-3,-3) ..controls +(3,.2)
and +(1,.7) .. (-1,-2)..controls +(-1,-.7) 
and +(-.1,-2) .. (.5,0) .. controls +(.1,2)
and +(1,.8) .. (-1.5,2.2) .. controls +(-1,-.8)
and +(1,-2) .. (-3,3) node[anchor=east]{$C$};
\draw[dashed] (-1,-2) -- (-1,-.2);
\draw[dashed] (-1.5,2.2) -- (-1.5,.2);
\draw (-1,-1.1) node[anchor=east]{$I_1$};
\draw (-1.5,1.2) node[anchor=east]{$I_2$};
\end{tikzpicture}
\caption{Comparison between an arbitrary Hankel-type contour $C$ and the standard ``keyhole'' $C_0$. $C_0$ coincides with the horizontal lines $\Im z=\pm\delta$ respectively when $t\to\pm\infty$, and circumvents the origin on a circular arc. } \label{fig:any-hankel}
\end{figure}
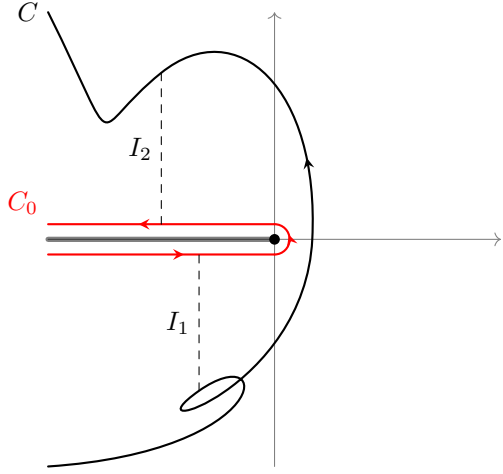

\end{proof}
\begin{remark}
The function $g$ in Proposition \ref{prop:hankel-int} may have finitely many singularities along the cut $(-\infty,0]$, but the order estimate must be uniform for $|z|$ sufficiently large. In particular, the integrands of \eqref{Zd-is-inv-lt} and \eqref{DZd-inv-lt} satisfy these conditions.
\end{remark}

\begin{figure}
\centering
\begin{tikzpicture}[decoration={markings, 
    mark=at position 0.6 with {\arrow{stealth}}}, line cap=round]
\draw[gray,->](-3,0)--(3,0);
\draw[line width=2pt,opacity=0.5](-3,0)--(0,0);
\draw[gray,->](0,-3)--(0,3);
\fill(0,0) circle[radius=2pt];
\draw[red,thick,postaction={decorate}] (-3,-.2)--({-.2*sqrt(3)},-.2);
\draw[red,thick,postaction={decorate}] ({-.2*sqrt(3)},-.2) arc[radius=.4,start angle=-150,end angle=150];
\draw[red,thick,postaction={decorate}] ({-.2*sqrt(3)},.2)--(-3,.2);
\draw[thick,postaction={decorate}] (1,-3)--(1,3);
\draw[dashed,postaction={decorate}] (0,2.5) -- (1,2.5)node[anchor=west]{$\lambda+(M+\delta)\ii$} ;
\draw (.5,2.8) node{$I_4$};
\draw[dashed,postaction={decorate}] (-2.5,.2)node[anchor=south east]{$-M+\delta\ii$} .. controls +(0,1.5) and +(-1.5,0) .. (0,2.5);
\draw (-2.1,2.1) node{$I_3$};
\draw[dashed,postaction={decorate}] (1,-2.5)node[anchor=west]{$\lambda-(M+\delta)\ii$} -- (0,-2.5);
\draw (.5,-2.8) node{$I_1$};
\draw[dashed,postaction={decorate}] (0,-2.5) .. controls +(-1.5,0) and +(0,-1.5) .. (-2.5,-.2) node[anchor=north east]{$-M-\delta\ii$};
\draw (-2.1,-2.1) node{$I_2$};
\end{tikzpicture}
\caption{Hankel-type, or ``keyhole'' contour (red), and its relation with the contour of inverse Laplace transform (black).} \label{fig:keyhole}
\end{figure}
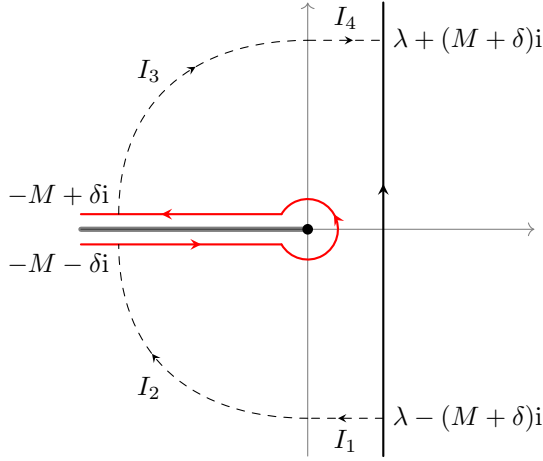

Then, we present an important observation that bends the contours of \eqref{Zd-is-inv-lt} and \eqref{DZd-inv-lt}.
\begin{theorem} \label{thm:DZ-keyhole-int}
For any $\mu\in\BbbR^d$ with $0=\mu_1\ge\mu_2\ge\cdots\ge\mu_d$, the path of integration in \eqref{Zd-is-inv-lt} and \eqref{DZd-inv-lt} can be changed to a Hankel-type contour $C$, that is:
\begin{align}
    Z(\mu) &= \frac{\Gamma(\frac{d}{2})}{2 \pi\ii} \oint_C \e^z \prod_{j=1}^d (z-\mu_j)^{-\frac12} \d z, \label{Zd-hankel} \\
    \D^\alpha Z(\mu) &= \frac{\Gamma(\frac{d}{2})}{2 \pi\ii} \prod_{j=1}^d \frac{(2\alpha_j-1)!!}{2^{\alpha_j}} \oint_C \e^z \prod_{j=1}^d (z-\mu_j)^{-\frac12-\alpha_j} \d z. \label{DZd-hankel}
\end{align}
\end{theorem}

\begin{proof}
It suffices to check \eqref{Zd-hankel}, as \eqref{DZd-hankel} follows directly from \eqref{Zd-hankel} by differentiation under the integrand, which is correct due to the analyticity of the integrand with respect to $\mu$ and $z$ as well as the exponential convergence of $\e^z$ as $\Re z\to -\infty.$ 

Assume for the moment that $d\ge 3$, and the integral \eqref{Zd-is-inv-lt} is absolutely convergent because the decay rate of the integrand $F(z)=\e^z \prod_j (z-\mu_j)^{-\frac12}$ is controlled by $|z|^{-\frac32}$.

Since the specific shape of the Hankel-type contour is irrelevant by Proposition \ref{prop:hankel-int}, we can choose a special contour consisting of two rays parallel to the real axis (with distance $\delta>0$) and a small circular arc around the origin, and choose an arbitrary $\lambda>0$ in \eqref{Zd-is-inv-lt}. 
Then, for given $M>0$, the finite truncations of the line integral (onto an interval $[-M-\delta,M+\delta]$) and the Hankel-type integral (onto the part with $\Re C(t)>-M$) differ by the combination of the following path integrals:
\begin{align*}
    I_1 = \int_{\lambda-(M+\delta)\ii}^{-(M+\delta)\ii} F(z)\d z,&\ 
I_2 = \int_{-(M+\delta)\ii}^{-M-\delta\ii} F(z)\d z,\\
I_3 = \int_{-M+\delta\ii}^{(M+\delta)\ii} F(z)\d z,&\
I_4 = \int_{(M+\delta)\ii}^{\lambda+(M+\delta)\ii} F(z)\d z.
\end{align*}
They are depicted in Figure \ref{fig:keyhole}. 

Since the line segments for $I_1,I_4$ have finite length and $F(z)\to 0$ as $|z|\to\infty$, we have that
\[\lim_{M\to\infty} |I_1| = \lim_{M\to\infty} |I_4| = 0.\]
Since $F(z)$ is analytic on the slit plane, the specific shape of the path of $I_3$ does not matter so long as it connects $(M+\delta)\ii$ and $-M+\delta\ii$. We choose its path to be a quarter circle with radius $M$, centred at $+\delta\ii$. Then, we estimate that
\[|I_3| \le \frac{\pi}{2}M \cdot \sup_{|z|\ge M-\delta} \prod_{j=1}^d |z-\mu_j|^{-\frac12} \lesssim M^{-\frac{d-2}{2}}\to 0. \]
The estimate for $|I_2|$ follows similarly. Combining the estimates above shows that
\[\int_{\lambda-\ii\infty}^{\lambda+\ii\infty} F(z)\d z = \oint_C F(z)\d z,\]
proving our claim for $d\ge 3$.

For $d=1$, the fact is trivial. First, $Z(\mu_1)=\e^{\mu_1}$ by the definition \eqref{Z-mu}, and by the gauge condition $\mu_1=0$ we have that $Z\equiv 1$. Then, the result \eqref{Zd-hankel} is equivalent to Hankel's formula for the $\Gamma$ function \cite{abramowitz_handbook_2013}:
\begin{equation}\label{gamma-hankel}
    \frac{1}{\Gamma(s)} = \frac{1}{2\pi\ii}\int_{C} z^{-s}\e^z\d z,\ \forall s\in\BbbC.
\end{equation}
Formulas for the derivatives follow similarly.

For $d=2$, we directly verify the identity
\begin{equation} \label{Zd-int-2d}
    Z(\mu_1,\mu_2) = \frac{1}{2\pi\ii} \oint_C \frac{\e^z}{\sqrt{(z-\mu_1)(z-\mu_2)}}\d z.
\end{equation}
We first notice that $Z(\mu)$ as defined by \eqref{Z-mu} is actually the modified Bessel function of the first kind \cite{abramowitz_handbook_2013}. Parametrizing $\BbbS^1$ by $m=[\cos\theta,\sin\theta]^T, \theta\in[0,2\pi)$, we get that
\begin{align*}
    Z(\mu) &= \frac{1}{2\pi} \int_0^{2\pi} \exp(\mu_1\cos^2\theta+\mu_2\sin^2\theta)\d\theta \\
    &= \e^{\frac{\mu_1+\mu_2}{2}} \frac{1}{2\pi}\int_0^{2\pi} \exp\left( \frac{\mu_1-\mu_2}{2}\cos 2\theta \right)\d \theta \\
    &= \e^{\frac{\mu_1+\mu_2}{2}} \frac{1}{2\pi}\int_0^{2\pi} \exp\left( \frac{\mu_1-\mu_2}{2}\cos t \right)\d t
    = \e^{\frac{\mu_1+\mu_2}{2}} \mathrm{I}_0\left( \frac{\mu_1-\mu_2}{2} \right).
\end{align*}
We also notice that the integrand $F(z)=\frac{\e^z}{\sqrt{(z-\mu_1)(z-\mu_2)}}$ is now analytic on the ray $(-\infty, \mu_2)$, since \emph{according to our square root convention, the branch cuts of $(z-\mu_1)^{-\frac12}$ and $(z-\mu_2)^{-\frac12}$ cancel out.} Therefore, we can actually convert the Hankel-type contour $C$ into a regular closed loop $C'$ (as shown in Figure \ref{fig:closed-contour}) without affecting the value of the integral.
Denoting by $a=\frac{\mu_1+\mu_2}{2}, b= \frac{\mu_1-\mu_2}{2}$, we rewrite the RHS of \eqref{Zd-int-2d} as
\[ \text{RHS} = \frac{1}{2\pi\ii}\oint_{C'} \frac{\e^z}{\sqrt{(z-a-b)(z-a+b)}}\d z= \frac{\e^a}{2\pi\ii}\oint_{|z|=R} \frac{\e^{z}}{\sqrt{z^2-b^2}}\d z,\]
where the loop $C'$ is further translated to the right by $a$, and deformed into a circle of radius $R>b$. We have also safely merged the two square roots in the denominator because the function is now a well-defined analytic function on the complex plane excluding the segment $(-b,b)$.
Expanding $\e^z$ into Taylor series, and using the Laurent series of $\frac{1}{\sqrt{z^2-b^2}}$ at $\infty$:
\[\frac{1}{\sqrt{z^2-b^2}} = \frac{1}{z} \qty(1-\frac{b^2}{z^2})^{-\frac12} = \sum_{k=0}^\infty \frac{(2k-1)!!}{2^k k!}\frac{b^{2k}}{z^{2k+1}},\ |z|>b,\]
we get
\begin{align*}
    \text{RHS} &= \e^a \sum_{k=0}^\infty \frac{1}{2\pi\ii} \frac{1}{k!} \oint_{|z|=R} \frac{z^k}{\sqrt{z^2-b^2}}\d z \\
    &=\e^a\sum_{k=0}^\infty \frac{b^{2k}}{4^k (k!)^2} = \e^a\mathrm{I}_0(b)
\end{align*}
by \cite{abramowitz_handbook_2013}, which is identical to the LHS.
\begin{figure}
\centering
\begin{tikzpicture}[decoration={markings, 
    mark=at position 0.6 with {\arrow{stealth}}}, line cap=round]
\draw[gray,->](-3,0)--(3,0);
\draw[gray,->](0,-1.2)--(0,1.2);
\draw[line width=2pt,opacity=0.5](-1,0)--(0,0);
\fill(0,0) circle[radius=2pt];
\fill(-1,0) circle[radius=2pt];
\draw[red,thick,dashed] (-3,-.2) .. controls +(.5,0) and +(-.5,.1) .. (-1.5,-.3);
\draw[red,thick,postaction={decorate}] (-1.5,-.3) .. controls +(.5,-.1) and +(-1,0) .. (0,-.8) .. controls +(.5,0) and +(0,-.5) .. (1,0);
\draw[red,thick,postaction={decorate}] (1,0) .. controls +(0,.5) and +(.5,0) .. (0,.8).. controls +(-1,0) and +(.5,.1) .. (-1.5,.3);
\draw[red,thick,postaction={decorate}] (-1.5,.3) -- (-1.5,-.3);
\draw[red,thick,dashed] (-1.5,.3).. controls +(-.5,-.1) and +(.5,0) .. (-3,.2);
\end{tikzpicture}
\caption{Closing the Hankel-type contour when $d=2$} \label{fig:closed-contour}
\end{figure}
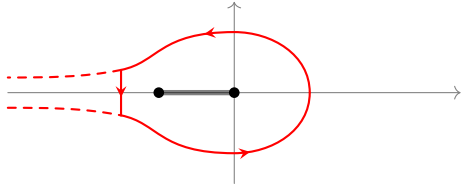
\end{proof}

The argument in the proof of Theorem \ref{thm:DZ-keyhole-int} for $d=2$ indicates that the contour of integration can also be simplified in general dimension $d$.
In fact, the integrand $F(z)=\e^z \prod_j (z-\mu_j)^{-\frac12}$ is analytic at $x\in \BbbR_-$ as long as there are an even number of negative square roots $\sqrt{x-\mu_j}$ with $x<\mu_j$, so that their branch cuts cancel out.
That is, $F(z)$ is analytic on the line segment $(\mu_{2k+1},\mu_{2k})$. The Hankel-type contour $C$ can then be split across these segments to form separate loops (see Figure \ref{fig:split-contour}).
\begin{figure}
\centering
\begin{tikzpicture}[decoration={markings, 
    mark=at position 0.55 with {\arrow{stealth}}}, line cap=round]
\draw[gray,->](-5,0)--(2,0);
\draw[gray,->](0,-1.5)--(0,1.5);
\draw[line width=2pt,opacity=0.5](-1,0)--(0,0);
\draw[line width=2pt,opacity=0.5](-3,0)--(-2,0);
\draw[line width=2pt,opacity=0.5](-5,0)--(-4,0);
\fill(0,0) circle[radius=2pt];
\fill(-1,0) circle[radius=2pt];
\fill(-2,0) circle[radius=2pt];
\fill(-3,0) circle[radius=2pt];
\fill(-4,0) circle[radius=2pt];
\draw[red,thick,postaction=decorate] (-5,-1)--(0,-1);
\draw[red,thick,postaction=decorate] (0,-1) arc[radius=1,start angle =-90,end angle=90];
\draw[red,thick,postaction=decorate] (0,1)--(-5,1);
\draw[red] (1,0) node[anchor=south west]{$C$};
\draw[blue,thick,postaction=decorate] (-1,-.3)--(0,-.3);
\draw[blue,thick,postaction=decorate] (0,-.3) arc[radius=0.3,start angle =-90,end angle=90];
\draw[blue,thick,postaction=decorate] (0,.3)--(-1,.3);
\draw[blue,thick,postaction=decorate] (-1,.3) arc[radius=0.3,start angle =90,end angle=270];
\draw[blue,thick,postaction=decorate] (-3,-.3)--(-2,-.3);
\draw[blue,thick,postaction=decorate] (-2,-.3) arc[radius=0.3,start angle =-90,end angle=90];
\draw[blue,thick,postaction=decorate] (-2,.3)--(-3,.3);
\draw[blue,thick,postaction=decorate] (-3,.3) arc[radius=0.3,start angle =90,end angle=270];
\draw[blue,thick,postaction=decorate] (-5,-.3)--(-4,-.3);
\draw[blue,thick,postaction=decorate] (-4,-.3) arc[radius=0.3,start angle =-90,end angle=90];
\draw[blue,thick,postaction=decorate] (-4,.3)--(-5,.3);
\draw[thick,dashed] (-1,-1)--(-1,-.3);
\draw[thick,dashed] (-1,1)--(-1,.3);
\draw[thick,dashed] (-2,-1)--(-2,-.3);
\draw[thick,dashed] (-2,1)--(-2,.3);
\draw[thick,dashed] (-3,-1)--(-3,-.3);
\draw[thick,dashed] (-3,1)--(-3,.3);
\draw[thick,dashed] (-4,-1)--(-4,-.3);
\draw[thick,dashed] (-4,1)--(-4,.3);
\end{tikzpicture}
\caption{Splitting the Hankel-type contour $C$ (in red) across the segments where $F(z)=\e^z \prod_j (z-\mu_j)^{-1/2}$ is analytic (the non-shaded ones) and forming small loops around every other segment (in blue). For even $d$, all such loops are finite; for odd $d$, the left-most loop is still infinite.}\label{fig:split-contour}
\end{figure}
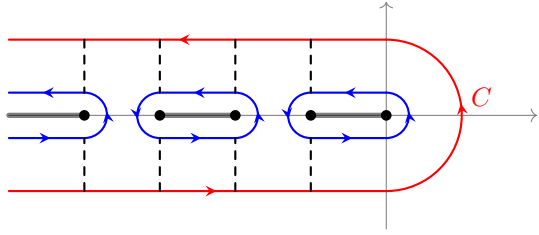
By shrinking the loops further towards the real axis, we can reduce it to a real integral, as shown in the following corollary.
\begin{corollary} \label{cor:Zd-real-int}
Work under the assumptions of Theorem \ref{thm:DZ-keyhole-int}. If we assume further that the eigenvalues are strictly separated, i.e.
\[0=\mu_1>\mu_2>\cdots>\mu_d,\]
then
\begin{equation} \label{Zd-real-int}
    Z(\mu) = \frac{\Gamma(\frac{d}{2})}{\pi} \sum_{k=1}^{\lceil d/2 \rceil}(-1)^{k-1} \int_{\mu_{2k}}^{\mu_{2k-1}} \frac{\e^x}{\sqrt{\prod_{j=1}^d |x-\mu_j|}}\d x,
\end{equation}
where $\mu_{d+1}\triangleq -\infty$.
\end{corollary}
\begin{proof}
Suppose that the Hankel-type contour $C$ consists of two rays parallel to the real axis with distance $\delta$ and a small arc around the origin. Then, we have that
\begin{equation}
    Z(\mu) = \frac{\Gamma(\frac{d}{2})}{2\pi\ii} \int_{-\infty}^0 [F(x-\delta\ii)-F(x+\delta\ii)]\d x + R,
\end{equation}
where $F(z) = \e^z \prod_j (z-\mu_j)^{-\frac12}$ is the integrand, and $R$ denotes the contribution from the small circular arc, which obviously vanishes as $\delta\to 0$.

We shrink $\delta\to 0$, so that the Hankel-type contour $C$ approaches the negative real axis. The limit of the integral can be taken inside the integral sign since both terms in the integrand are dominated by
\[\e^x\prod_{j=1}^d |x-\mu_j|^{-\frac12} \in L^1(-\infty,0).\]
If $x\in (\mu_{2k+1},\mu_{2k})$ (i.e.~there are an even number of eigenvalues larger than $x$), then the limits of $F(z)$ at $x$ from above and below are identical by the definition of the square root, so they cancel out. 
Otherwise, if $x\in(\mu_{2k},\mu_{2k-1})$, then the limits of $F(x\pm\delta\ii)$ are
\[\lim_{\delta\to 0} F(x\pm\delta \ii) = (\mp\ii)^{2k-1}\e^x  \prod_{j=1}^d |x-\mu_j|^{-\frac12}
=\pm (-1)^k \ii \e^x  \prod_{j=1}^d |x-\mu_j|^{-\frac12},\]
as $z^\frac12\to \ii\sqrt{|x|}$ as $z\to x$ above the real axis ($-\ii\sqrt{|x|}$ below the real axis) for any $x<0$. Substituting the limits above, we get that
\begin{align*}
    Z(\mu) &= \frac{\Gamma(\frac{d}{2})}{2\pi\ii} \sum_{k=1}^{\lceil d/2 \rceil} \int_{\mu_{2k}}^{\mu_{2k-1}} (-2\ii) (-1)^k \e^x \prod_{j=1}^d |x-\mu_j|^{-\frac12}\d x \\
    &= \frac{\Gamma(\frac{d}{2})}{\pi} \sum_{k=1}^{\lceil d/2 \rceil} (-1)^{k-1} \int_{\mu_{2k}}^{\mu_{2k-1}} \frac{\e^x }{\sqrt{\prod_{j=1}^d |x-\mu_j|}}\d x
\end{align*}
as desired.
\end{proof}

\section{Asymptotic expansions}

In this section, we discuss the asymptotic properties of $Z$ and its derivatives when $\mu_1=0$ and one or more $\mu_j\to-\infty$, which was first studied by \cite{bingham_distribution_1964,kent_asymptotic_1987} in three dimensions. The integral form obtained in the previous section makes it convenient to obtain asymptotic series with a technique similar to \cite{mcclure_asymptotic_1987}.

As we will observe, the function $Z$ from different dimensions will naturally emerge as coefficients in the expansion, so we will update the notation to $Z^{(d)}(\mu)$ to emphasize the dimensionality of $\mu$.

\subsection{Preliminaries} We begin with a formal derivation of an asymptotic series of $Z^{(d)}(\mu)$ when one eigenvalue $\mu_d\to-\infty$.
In the integral representation of $Z^{(d)}(\mu)$ \eqref{Zd-hankel}, we can reformulate the term $(z-\mu_d)^{-\frac12}$ as follows:
\begin{equation} \label{pow-expand}
    \begin{aligned}
        (z-\mu_d)^{-\frac12} &= (-\mu_d)^{-\frac12} \qty(1+ \frac{z}{-\mu_d})^{-\frac12} \\
        &= \frac{1}{\sqrt{|\mu_d|}} \bigg[1 - \frac12 \frac{z}{-\mu_d} + \frac38 \frac{z^2}{(-\mu_d)^2}+ \cdots \\
        &\hspace{50pt} + \frac{(-1)^n (2n-1)!!}{(2n)!!} \frac{z^n}{(-\mu_d)^n} + O\qty(\frac{1}{|\mu_d|^{n+1}}) \bigg].
    \end{aligned}
\end{equation}
Substituting it in, we find that \small
\begin{align*}
    Z^{(d)} &= \frac{\Gamma(\frac{d}{2})}{2\pi\ii} \oint_C \e^z \prod_{j=1}^{d-1}(z-\mu_j)^{d-1} \frac{1}{\sqrt{|\mu_d|}}\left[ \sum_{k=0}^n \frac{(-1)^k(2k-1)!!}{(2k)!!} \frac{z^k}{(-\mu_d)^k} + O\qty(\frac{1}{|\mu_d|^{n+1}}) \right] \d z\\
    &=\frac{1}{\sqrt{|\mu_d|}}
    \sum_{k=0}^n \frac{(-1)^k(2k-1)!!}{(2k)!! (-\mu_d)^k}
    \frac{\Gamma(\frac{d}{2})}{2\pi\ii} \oint_C z^k\e^z \prod_{j=1}^{d-1}(z-\mu_j)^{d-1} \d z 
    +O\qty(\frac{1}{|\mu_d|^{n+\frac32}}).
\end{align*}\normalsize
The coefficient with $k=0$ resembles the $(d-1)$-dimensional normalizing constant, while the remaining coefficients are closely related contour integrals. This reflects the dimensional hierarchy in Bingham distributions: as $\mu_d\to-\infty$, the weight suppresses $m_d\neq 0$ and the distribution concentrates on the equatorial circle $m_d=0$.



To formulate the expansion and control its remainder, we introduce the auxiliary quantities
\begin{equation} \label{Mdk}
    M^{(d)}_k(\mu_1,\ldots,\mu_d) = \frac{1}{2\pi\ii} \oint_C z^k \e^z \prod_{j=1}^d (z-\mu_j)^{-\frac12} \d z.
\end{equation}
By \eqref{Zd-hankel},
\begin{equation} \label{Zd-is-Md0}
    Z^{(d)}=\Gamma(\tfrac{d}{2})M^{(d)}_0,
\end{equation}
so $M^{(d)}_k$ serves as a generalization of the Bingham normalizing constant $Z^{(d)}$.
In fact, $M^{(d)}_k$ is closely associated with $Z^{(d)}$.

\begin{proposition} \label{prop:Mdk}
Assume that $0=\mu_1\ge\mu_2\ge\cdots\ge\mu_d$.
Then, $M^{(d)}_k$ can be explicitly written as
\begin{equation} \label{Mdk-expr}
    M^{(d)}_k(\mu) = \frac{k!}{2\pi^{\frac{d}{2}}} \sum_{l=0}^k \frac{1}{(k-l)!} \binom{\frac{d}{2} - 1}{l}
    \int_{\BbbS^{d-1}} \e^{\sum_{j=1}^d \mu_j m_j^2} \qty(\sum_{j=1}^d \mu_j m_j^2)^{k-l} \d m,
\end{equation}
where $\binom{\alpha}{n}=\frac{\alpha(\alpha-1)\cdots(\alpha-n+1)}{n!}$ is the generalized binomial coefficient. As a result, all $M^{(d)}_k$ are real numbers.
\end{proposition}
\begin{proof}
Let the contour $C$ go horizontally from $-\infty-\ii$ to $-\ii$, around the origin on a semicircle with radius 1 up to $\ii$, and then back to $-\infty+\ii$ horizontally again (see the outer loop in Figure \ref{fig:split-contour}). 
We consider the moment generating function of $M^{(d)}_k$:
\begin{equation} \label{Mdk-mgf}
    G(t) = \sum_{k=0}^\infty \frac{M^{(d)}_k t^k}{k!}
    =\sum_{k=0}^\infty \frac{1}{2\pi\ii} \oint_C \frac{(tz)^k}{k!} \e^z \prod_{j=1}^d(z-\mu_j)^{-\frac12} \d z
\end{equation}
The sum of the absolute values of the terms in the integrand is bounded by
\[ |\e^z| \cdot \prod_{j=1}^d|z-\mu_j|^{-\frac12} \cdot \sum_{k=0}^\infty\frac{|tz|^k}{k!} = \e^{\Re z + |tz|}\prod_{j=1}^d |z-\mu_j|^{-\frac12},\]
which is still exponentially convergent as $\Re z\to-\infty$ on $C$, provided that $|t|\ll 1$. 
Using the dominated convergence theorem, we can then exchange summation and integration to get
\begin{align*}
    G(t) &= \frac{1}{2\pi\ii}\oint_C \e^z \sum_{k=0}^\infty \frac{(tz)^k}{k!} \prod_{j=1}^d(z-\mu_j)^{-\frac12} \d z
    =\frac{1}{2\pi\ii} \oint_C \e^{(1+t)z} \prod_{j=1}^d(z-\mu_j)^{-\frac12} \d z \\
    &= (1+t)^{\frac{d}{2}-1} \frac{1}{2\pi\ii}\oint_{C'} \e^z \prod_{j=1}^d\qty(z-(1+t)\mu_j)^{-\frac12}\d z,
\end{align*}
where in the final line, we assumed that $t$ is real, and changed the integration variable to $(1+t)z$. The new contour $C'=(1+t)C$ (point-wise dilation by $1+t$) is still a Hankel-type contour, so we can apply the formula \eqref{Zd-hankel} by Theorem \ref{thm:DZ-keyhole-int} to get the explicit expression
\begin{equation} \label{Mdk-mgf-expr}
    G(t) = \frac{(1+t)^{\frac{d}{2}-1}}{\Gamma(\frac{d}{2})} Z^{(d)}((1+t)\mu).
\end{equation}

Due to the convergence of the power series \eqref{Mdk-mgf}, its coefficients $M^{(d)}_k$ can be obtained by differentiation of \eqref{Mdk-mgf-expr} with respect to $t$ at $t=0$. Using \eqref{Z-mu} to express $Z^{(d)}(\mu)$, we get
\begin{align*}
    M^{(d)}_k(\mu) &=
    \frac{\d^k G(t)}{\d t^k} \bigg|_{t=0}
    =\frac{1}{\omega_d\Gamma(\frac{d}{2})} \int_{\BbbS^{d-1}} \frac{\d^k}{\d t^k} \left[ (1+t)^{\frac{d}{2}-1}\e^{(1+t) \sum_{j=1}^d \mu_j m_j^2} \right]_{t=0} \d m \\
    &=\frac{1}{2\pi^{\frac{d}{2}}}
    \int_{\BbbS^{d-1}} \e^{\sum_{j=1}^d \mu_j m_j^2} \sum_{l=0}^k \binom{k}{l}\qty(\sum_{j=1}^d \mu_j m_j^2)^{k-l} \cdot \qty(\frac{d}{2}-1)\cdots\qty(\frac{d}{2}-l) \d m,
\end{align*}
which simplifies into \eqref{Mdk-expr}.
\end{proof}

\subsection{Asymptotic expansion}

Now that we know about the auxiliary functions $M^{(d)}_k$, we discuss the asymptotic property of $Z^{(d)}$ as one or more eigenvalues $\mu_j\to -\infty.$ It is in fact more convenient to consider all the $M^{(d)}_k$ simultaneously.

The estimation requires a set of elementary lemmas. The first lemma is an estimate on the ratio between rising products and factorials.

\begin{lemma} \label{lem:pfact-ratio}
For any $p\ge 0$ and $n\in\BbbN$,
\begin{equation}
    \frac{p(p+1)\cdots(p+n-1)}{n!} \le K_1 (n+1)^{p-1},
\end{equation}
where the constant $K_1=K_1(p)$, and the LHS is regarded as 1 when $n=0$.
\end{lemma}
\begin{proof}
When $n=0$, $K_1(p)=1$ fulfills the estimate. 
When $n\ge 1$, we use the inequality $1+x\le \e^x, x\ge -1$ to get
\[ \text{LHS} = \prod_{j=1}^{n} \qty(1+\frac{p-1}{j})
    \le \prod_{j=1}^{n} \e^{\frac{p-1}{j}} = \exp((p-1)\sum_{j=1}^{n}\frac1j).\]
As $\{ \sum_{k=1}^n \frac1k - \ln (n+1)\}$ increases monotonically to the Euler-Mascheroni constant $c\approx 0.577\cdots$, we have that
\[ \ln (n+1)+(1-\ln 2) \le 1+\frac12+\cdots+\frac{1}{n} < \ln (n+1)+c,\]
so
\begin{align*}
    \text{LHS} &\le \exp[(p-1)\ln(n+1) + \max\qty{c(p-1), (1-\ln 2)(p-1)}] \\
    &=\max\{\e^{c(p-1)}, \e^{(1-\ln2)(p-1)}\} (n+1)^{p-1}.
\end{align*}
Therefore, we may choose $K_1(p)=\max\{\e^{c(p-1)}, 1\}$ as a uniform constant valid for all $p$ and $n$.
\end{proof}

The second lemma is an estimate on the Taylor remainder of the function $\prod_{j=1}^s (z-\nu_j)^{-\frac12}$ with respect to $z$, whose special case $s=1$ was used in the expansion \eqref{pow-expand}.

\begin{lemma} \label{lem:pow-taylor-err-mult}
Let $\vec a=[a_1,\ldots,a_s]^T, \vec p=[p_1,\ldots,p_s]^T$ be vectors of positive real numbers. Denote by
\begin{equation}
    \Phi_{\vec p}(z;\vec a)=\prod_{j=1}^s \qty(1+\frac{z}{a_j})^{-p_j}
\end{equation}
and
\begin{equation}
    r_{\vec p;n}(z;\vec a) = \Phi_{\vec p}(z;\vec a) - \sum_{l=0}^n \frac{z^l}{l!}\qty[\frac{\d^l}{\d z^l} \Phi_{\vec p}(z)]_{z=0}
\end{equation}
the Taylor remainder of $\Phi_{\vec p}$ at $z=0$ (we also abbreviate them as $\Phi_{\vec p}(z)$ and $r_{\vec p;n}(z)$). Then, we have that
\begin{equation} \label{pow-taylor-err-mult}
    |r_{\vec p;n}(z)| \le \frac{K_1(\vec p) (n+2)^{P-1} }{(\sin\theta)^{n+P+1} } |z|^{n+1} \sum_{|\beta|=n+1} \frac{1}{\vec a^\beta}
\end{equation}
uniformly for all $|{\arg z}|\le \pi-\theta$ ($\theta\in(0,\frac{\pi}{2})$), where
\begin{equation}
    K_1(\vec p)=\prod_{j=1}^s K_1(p_j)
\end{equation}
depends on $\vec p$ only, $P\triangleq p_1+\cdots+p_s$ and $\beta=(\beta_1,\ldots,\beta_s)\in\BbbN^s$ denotes a multi-index.
\end{lemma}
\begin{proof}
By Leibniz's formula, we have that
\begin{equation} \label{pow-mult-der}
\begin{aligned}
    \frac{\d^l}{\d z^l} \Phi_{\vec p}(z) &= \sum_{|\beta|=l} \frac{l!}{\beta!} \prod_{j=1}^s \frac{\d^{\beta_j}}{\d z^{\beta_j}}\qty(1+\frac{z}{a_j})^{-p_j}\\
&= (-1)^l \sum_{|\beta|=l} \frac{l!}{\beta!} \prod_{j=1}^s \frac{p_j(p_j+1)\cdots(p_j+\beta_j-1)}{a_j^{\beta_j}}\qty(1+\frac{z}{a_j})^{-p_j-\beta_j}.
\end{aligned}
\end{equation}
Using the integral representation of the remainder with the path being the segment connecting $0$ and $z$, we have that\small
\begin{align*}
    |r_{\vec p;n}(z)|
    &=\left| \frac{1}{n!} \int_0^z (z-\zeta)^n\frac{\d^{n+1}}{\d \zeta^{n+1}}\Phi_{\vec p}(\zeta) \d\zeta \right| \\
    &\le (n+1)\int_0^z |z-\zeta|^n |{\d \zeta}| \sum_{|\beta|=n+1} \prod_{j=1}^s \frac{p_j(p_j+1)\cdots(p_j+\beta_j-1)}{\beta_j! a_j^{\beta_j}} \left|1+\frac{\zeta}{a_j}\right|^{-p_j-\beta_j}.
\end{align*}\normalsize
We use Lemma \ref{lem:pfact-ratio} and some basic arithmetic to estimate the rising products:
\begin{align*}
    \prod_{j=1}^s  \frac{p_j(p_j+1)\cdots(p_j+\beta_j-1)}{\beta_j! a_j^{\beta_j}} 
    &\le \prod_{j=1}^s \frac{K_1(p_j)(\beta_j+1)^{p_j-1}}{a_j^{\beta_j}} 
    = \frac{K_1(\vec p)}{\vec a^{\beta}}\cdot  \frac{ \prod_{j=1}^s(\beta_j+1)^{p_j}}{ \prod_{j=1}^s (\beta_j+1)} \\
    &\le \frac{K_1(\vec p)}{\vec a^{\beta}} \cdot \frac{(n+2)^P}{1+\sum_{j=1}^s \beta_j} =\frac{K_1(\vec p) (n+2)^{P-1}}{\vec a^{\beta}}.
\end{align*}
We also use the geometric relation that $|{\arg\zeta}|\le \pi-\theta\Rightarrow |1+\frac{\zeta}{a_j}| \ge \sin\theta$ to get
\begin{align*}
    \int_0^z |z-\zeta|^n |{\d \zeta}| \prod_{j=1}^s \left|1+\frac{\zeta}{a_j}\right|^{-p_j-\beta_j}
&\le \frac{1}{(\sin\theta)^{|\beta|+P}}\int_0^{|z|} (|z|-t)^n \d t\\
&= \frac{|z|^{n+1}}{(n+1)(\sin\theta)^{n+P+1}}.
\end{align*}
Combining these estimates, we obtain
\[|r_{\vec p;n}(z)| \le \frac{K_1(\vec p) (n+2)^{P-1}}{(\sin\theta)^{n+P+1}} |z|^{n+1} \sum_{|\beta|=n+1} \frac{1}{\vec a^{\beta}},\]
which is exactly as predicted in \eqref{pow-taylor-err-mult}.
\end{proof}

To represent coefficients of multivariate series more compactly, we introduce the following notations for multi-index $\alpha\in\BbbR^s$:
\begin{subequations} \label{multi-index}
\begin{itemize}
\item Scalar addition: for $c\in\BbbR$, we write
\begin{equation}
    \alpha+c \triangleq (\alpha_1+c,\ldots,\alpha_s+c).
\end{equation}
\item Factorials and double factorials: if $\alpha\in\BbbN^s$, we define
\begin{equation}
    \alpha! \triangleq \prod_{j=1}^s \alpha_j!,\
    \alpha!! \triangleq \prod_{j=1}^s \alpha_j!!,
\end{equation}
where $0!=0!!=(-1)!!=1$ by convention.
\item Monomial: for $x=(x_1,\ldots,x_s)\in\BbbR^s$, we write
\begin{equation}
    x^\alpha \triangleq \prod_{j=1}^s x_j^{\alpha_j}.    
\end{equation}
\end{itemize}
\end{subequations}

To describe the decay rate with respect to the variables $\mu_j\to -\infty$, we define the modulus
\begin{equation} \label{w-alpha}
    \llbracket x \rrbracket_\alpha = \prod_{j=1}^s \max\{1,|x_j|\}^{-\alpha_j}
\end{equation}
for $x\in\BbbR^s, \alpha\in(\BbbR_+)^s$. The function $\llbracket \cdot\rrbracket_\alpha$ is strictly positive, uniformly bounded by 1 and shrinks to 0 when any $x_j\to -\infty$. 
If $\alpha=a\vec 1=[a,\ldots,a]^T$ is composed of the same number $a$, we also abbreviate $\llbracket \cdot\rrbracket_\alpha$ as $\llbracket \cdot\rrbracket_a$ for simplicity.
An obvious property is that $\llbracket \cdot\rrbracket_\alpha$ factorizes into components: if we partition $x$ into $x=\smqty[x^{(1)}\\x^{(2)}]$ ($x^{(1)}\in \BbbR^t, x^{(2)}\in\BbbR^{s-t}$), and $\alpha$ accordingly into $\smqty[\alpha^{(1)}\\ \alpha^{(2)}]$, then
\begin{equation} \label{w-alpha-div}
    \llbracket x \rrbracket_\alpha = 
    \llbracket x^{(1)} \rrbracket_{\alpha^{(1)}} \llbracket x^{(2)} \rrbracket_{\alpha^{(2)}}.
\end{equation}

We are now ready for the asymptotic expansion when one or more eigenvalues of $\mu$ are sent to $-\infty$.

\begin{theorem} \label{thm:Mdk-asymp-mult}
Assume that $0=\mu_1\ge\mu_2\ge\cdots\ge\mu_d$ and that $1\le s<d$, and denote by $\mu=\mqty[\mu'\\\nu]$, with $\mu'=[\mu_1,\ldots,\mu_{d-s}]^T\in\BbbR^{d-s}$ and $\nu=[\mu_{d-s+1},\ldots,\mu_d]^T=[\nu_1,\ldots,\nu_s]^T\in \BbbR^s$.
If $\nu_1\le -1$, then $M^{(d)}_k$ has the asymptotic expansion
\begin{equation} \label{Mdk-asymp-mult}
    M^{(d)}_k(\mu) = \llbracket \nu \rrbracket_{\frac12} \sum_{|\beta|\le n} \frac{(2\beta-1)!!}{(-2)^{|\beta|} \beta!} \frac{ M^{(d-s)}_{k+|\beta|}(\mu')}{(-\nu)^\beta} + R^{(d,s)}_{k,n}(\mu),
\end{equation}
where $\beta\in\BbbN^s$ is a multi-index. 
The remainder has an estimate
\begin{equation} \label{Mdk-mult-rem}
    \left|R^{(d,s)}_{k,n}\right| \le K_2(n,k,d,s) \llbracket\mu\rrbracket_{\frac12}\sum_{|\beta|=n+1} (-\nu)^{-\beta}
\end{equation}
uniformly for $\nu_1\le -1$,
where the constant $K_2$ has an explicit expression
\begin{equation} \label{K2-const}
    K_2 = \frac{6\sqrt{2}\e}{\pi} (n+2)^{\frac{s}{2}-1} 2^{n+\frac{k}{2}+\frac{d}{4}} (n+k+1)!.
\end{equation}
\end{theorem}
\begin{remark}
The estimate \eqref{Mdk-mult-rem} is uniform in $\mu$: boundedness assumptions are imposed on neither $\mu'$ nor $\nu$. If components of $\mu'$ also diverge, the coefficients have their own asymptotic expansions, and a different choice of $s$ may be more effective. 
As usual, the factorial growth of the remainder constant does not imply convergence as $n\to\infty$; rather, the expansion is asymptotic in $\nu_j\to-\infty$.
\end{remark}
\begin{proof}
We choose a Hankel-type contour $C_1$ with the following parametrization (see Figure \ref{fig:contour1}):
\begin{equation} \label{contour-C1}
C_1(t) = \begin{cases}
    -\ii+\e^{\ii\theta}(t+1), & t<-1, \\
    \e^{\frac{\ii\pi}{2}t}, & -1\le t\le 1, \\
    \ii-\e^{-\ii\theta}(t-1), & t>1.
\end{cases}
\end{equation}
Geometrically, it is composed of an arc with radius 1 connecting $\pm\ii$, and two rays going to $-\infty$ at an angle $\theta\in(0,\frac{\pi}{2})$. One of the rays starts from $\ii$ with slope $-\tan\theta$, and the other starts at $-\ii$ with slope $\tan\theta$. It matches the definition of a Hankel-type contour since its imaginary part is linear in its real part.

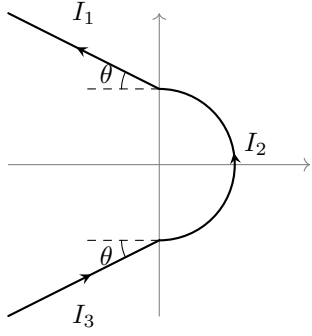
\begin{figure}
\centering
\begin{tikzpicture}[decoration={markings, 
    mark=at position 0.55 with {\arrow{stealth}}}, line cap=round]
\draw[gray,->](-2,0)--(2,0);
\draw[gray,->](0,-2)--(0,2);
\draw[thick,postaction=decorate] (-2,-2)--(0,-1);
\draw[thick,postaction=decorate](0,-1) arc[start angle=-90, end angle=90, radius=1];
\draw[thick,postaction=decorate](0,1)-- (-2,2);
\draw (-1,2) node{$I_1$};
\draw (1,0) node[anchor=south west]{$I_2$};
\draw (-1,-2) node{$I_3$};
\draw[dashed] (0,1)--(-1,1);
\draw[dashed] (0,-1)--(-1,-1);
\draw (-.5,1) arc[radius=0.5,start angle=180,end angle=154];
\draw (-.7,1.2) node{$\theta$};
\draw (-.7,-1.2) node{$\theta$};
\draw (-.5,-1) arc[radius=0.5,start angle=180,end angle=206];
\end{tikzpicture}
\caption{The contour $C_1$ \eqref{contour-C1}} \label{fig:contour1}
\end{figure}

Similar to \eqref{pow-expand}, we rewrite the integrand of $M^{(d)}_k$ \eqref{Mdk} as
\[ M^{(d)}_k(\mu) = \frac{1}{2\pi\ii}\oint_{C_1} z^k\e^z\prod_{j=1}^{d-s} (z-\mu_j)^{-\frac12} \frac{1}{\sqrt{|\nu_1\cdots\nu_s|}} \cdot \uline{\prod_{j=1}^s \qty(1+\frac{z}{-\nu_j})^{-\frac12}} \d z. \]
The underlined part matches the form of $\Phi_{\vec p}(z)$ in Lemma \ref{lem:pow-taylor-err-mult} with $\vec p=(\frac12,\ldots,\frac12)$, so we expand this function with respect to $z$, and use that lemma to estimate the remainder.
Using the expression \eqref{pow-mult-der} to evaluate the derivatives of $\prod_j (1+\frac{z}{-\nu_j})^{-\frac12}$, we get\small
\begin{equation} \label{Mdk-expand-1}
\begin{aligned}
    M^{(d)}_k(\mu)
    &=\frac{1}{2\pi\ii} \oint_{C_1} z^k \e^z \prod_{j=1}^{d-s}(z-\mu_j)^{-\frac12} \frac{1}{\sqrt{|\nu_1\cdots\nu_s|}} \\
    &\hspace{60pt}\cdot \left[ \sum_{l=0}^n \frac{(-1)^l z^l}{l!}\sum_{|\beta|=l} \frac{l!}{\beta!} \prod_{j=1}^s \frac{\frac12\cdot \frac32\cdots (\beta_j-\frac12)}{(-\nu_j)^{\beta_j}} \right] \d z \\ 
    &\quad + \frac{1}{2\pi\ii} \oint_{C_1} z^k\e^z \prod_{j=1}^{d-s}(z-\mu_j)^{-\frac12} \frac{1}{\sqrt{|\nu_1\cdots\nu_s|}}r_{\frac12,\ldots,\frac12;n}(z;-\nu) \d z.
\end{aligned}
\end{equation}\normalsize
We focus on the summation first.
Extract the common factor $\frac{1}{\sqrt{|\nu_1\cdots\nu_s|}}$, which equals $\llbracket \nu \rrbracket_{\frac12}$ by \eqref{w-alpha} (since $|\nu_j|\ge 1$), and translate the summands into the multi-index notations \eqref{multi-index}, obtaining
\[\frac{\llbracket \nu \rrbracket_{\frac12}}{2\pi\ii} 
\oint_{C_1}  z^k \e^z \prod_{j=1}^{d-s}(z-\mu_j)^{-\frac12} \qty[\sum_{l=0}^n \sum_{|\beta|=l} z^l \frac{(-1)^l(2\beta-1)!!}{2^{|\beta|} \beta! (-\nu)^\beta} ]\d z. \]
Then, we reorganize the sum in terms of the multi-index $\beta\in\BbbN^s$, which traverses all indices $|\beta|\le n$, and extract all constants from the integral, transcribing it into
\begin{align*}
    &\ \ \llbracket \nu \rrbracket_{\frac12} \sum_{|\beta|\le n} \frac{(2\beta-1)!!}{(-2)^{|\beta|} \beta! (-\nu)^\beta} \cdot \frac{1}{2\pi\ii} \oint_{C_1} z^{k+|\beta|} \e^z \prod_{j=1}^{d-s}(z-\mu_j)^{-\frac12}\d z \\
    &=\llbracket \nu \rrbracket_{\frac12} \sum_{|\beta|\le n} \frac{(2\beta-1)!!}{(-2)^{|\beta|} \beta!} \frac{M^{(d-s)}_{k+|\beta|}(\mu')}{(-\nu)^\beta},
\end{align*}
by \eqref{Mdk} in dimension $d-s$, which matches the series in \eqref{Mdk-asymp-mult} exactly.

Hence, the remainder term $R^{(d,s)}_{k,n}$ corresponds to the second term in \eqref{Mdk-expand-1}:
\begin{equation} \label{rem-dskn-expr}
    R^{(d,s)}_{k,n}(\mu)
    =\frac{\llbracket \nu \rrbracket_{\frac12}}{2\pi\ii} \oint_{C_1} z^k\e^z \prod_{j=1}^{d-s}(z-\mu_j)^{-\frac12} r_{\frac12,\ldots,\frac12;n}(z;-\nu) \d z .
\end{equation}
By the geometric property of the contour $C_1$, we have that
\[|z-\mu_j| \ge \dist(\mu_j, C_1) \ge \max\{1, |\mu_j|\sin\theta\} \ge \sin\theta \max\{1,|\mu_j|\},\]
and thus,
\begin{equation} \label{sqrt-dist-on-C1}
    \prod_{j=1}^{d-s} |z-\mu_j|^{-\frac12} \le \prod_{j=1}^{d-s}\frac{\max\{1,|\mu_j|\}^{-\frac12}}{(\sin\theta)^{\frac12}} = \frac{\llbracket \mu' \rrbracket_{\frac12}}{(\sin\theta)^{\frac{d-s}{2}}}
\end{equation}
by \eqref{w-alpha}.
It is also obvious that $|\arg(\frac{z}{-\nu_j})|<\pi-\theta$ for all $z\in C_1$ and $\nu_j<0$, so the estimate in Lemma \ref{lem:pow-taylor-err-mult} holds for the Taylor remainder:
\begin{equation} \label{pow-rem-on-C1}
    \left| r_{\frac12,\ldots,\frac12;n} (z;-\nu)\right| \le \frac{[K_1(\frac12)]^s (n+2)^{\frac{s}{2}-1}}{(\sin\theta)^{n+\frac{s}{2}+1}} |z|^{n+1} \sum_{|\beta|=n+1} \frac{1}{(-\nu)^\beta}.
\end{equation}
Applying \eqref{sqrt-dist-on-C1} and \eqref{pow-rem-on-C1} on \eqref{rem-dskn-expr}, we have that
\begin{align*}
    |R^{(d,s)}_{k,n}| &\le \frac{\llbracket \nu \rrbracket_{\frac12}}{2\pi}
    \cdot \frac{\llbracket \mu' \rrbracket_{\frac12}}{(\sin\theta)^{\frac{d-s}{2}}} 
    \cdot \frac{[K_1(\frac12)]^s (n+2)^{\frac{s}{2}-1}}{(\sin\theta)^{n+\frac{s}{2}+1}} \sum_{|\beta|=n+1} 
    \frac{1}{(-\nu)^{\beta}} 
    \oint_{C_1} |z|^{k+n+1}|\e^z| |{\d z}| \\
    &=\frac{[K_1(\frac12)]^s (n+2)^{\frac{s}{2}-1}}{2\pi (\sin\theta)^{n+\frac{d}{2}+1} } \llbracket\mu\rrbracket_{\frac12} \sum_{|\beta|=n+1} \frac{1}{(-\nu)^{\beta}} \cdot  \uline{\oint_{C_1} |z|^{k+n+1}|\e^z| |{\d z}|},
\end{align*}
where we used the factorization property \eqref{w-alpha-div} to write $\llbracket\mu\rrbracket_{\frac12}=\llbracket \nu \rrbracket_{\frac12}\llbracket \mu' \rrbracket_{\frac12}$.
Hence, the final step is to estimate the underlined integral above, which is analogous to Hankel's formula \eqref{gamma-hankel} of the $\Gamma$ function. The following lemma provides the required estimate.

\begin{lemma}\label{lem:int-C1-est}
\begin{equation}
    \oint_{C_1} |z|^n |\e^z| |{\d z}| \le \frac{6\e\cdot n!}{(\cos\theta)^{n+1}}.
\end{equation}
\end{lemma}
\begin{proof}[Proof of Lemma \ref{lem:int-C1-est}]
We divide the integral into three parts: $I_1$ equals the integration from $\ii$ to $-\infty$ (above real axis), $I_3$ from $-\infty$ (below the axis) to $-\ii$, and $I_2$ on the semicircle (see annotations in Figure \ref{fig:contour1}).
Parametrizing the ray starting at $\ii$ by $z=\ii + (-1+\ii\tan\theta) t,$ ($t\in[0,\infty)$), we get that
\begin{align*}
    |I_1| &\le \int_0^\infty (1+ t\sqrt{1+\tan^2\theta})^{n} \cdot\e^{-t} \cdot \sqrt{1+\tan^2\theta}\d t \\
    &\le \frac{1}{(\cos\theta)^{n+1}}\int_0^\infty \qty(t+\cos\theta)^{n} \e^{-t}\d t \\
    &=\frac{\e^{\cos\theta}}{(\cos\theta)^{n+1}} \int_{\cos\theta}^\infty t^{n} \e^{-t}\d t\le \frac{\e \cdot n!}{(\cos\theta)^{n+1}}.
\end{align*}
By symmetry, the same estimate applies to $|I_3|$. On the semicircle, we have that $|z|\le 1$ and $|\e^z|\le \e,$ so
\[|I_2| \le \pi \e < 4\e \cdot \frac{n!}{(\cos\theta)^{n+1}}.\]
Combining these estimates concludes the lemma.
\end{proof}

With Lemma \ref{lem:int-C1-est} established, we eventually obtain the bound
\begin{align*}
    |R^{(d,s)}_{k,n}| &\le 
    \frac{[K_1(\frac12)]^s (n+2)^{\frac{s}{2}-1}}{2\pi (\sin\theta)^{n+\frac{d}{2}+1} } \llbracket\mu\rrbracket_{\frac12} \sum_{|\beta|=n+1} \frac{1}{(-\nu)^{\beta}}  
    \cdot \frac{6\e (n+k+1)!}{(\cos\theta)^{n+k+2}}\\
    &= \frac{3\e[K_1(\frac12)]^s}{\pi \sin\theta\cos^2\theta} \frac{(n+2)^{\frac{s}{2}-1} (n+k+1)!}{(\sin\theta)^{n+\frac{d}{2}}(\cos\theta)^{n+k}} \cdot \llbracket\mu\rrbracket_{\frac12}\sum_{|\beta|=n+1} \frac{1}{(-\nu)^{\beta}}.
\end{align*}
Finally, we let $\theta=\frac{\pi}{4}$ to obtain the estimate \eqref{Mdk-mult-rem}, and the constant expression \eqref{K2-const}.
\end{proof}

\begin{figure}
\centering
\includegraphics[width=\textwidth,trim={2cm 1.3cm 2cm 1cm}]{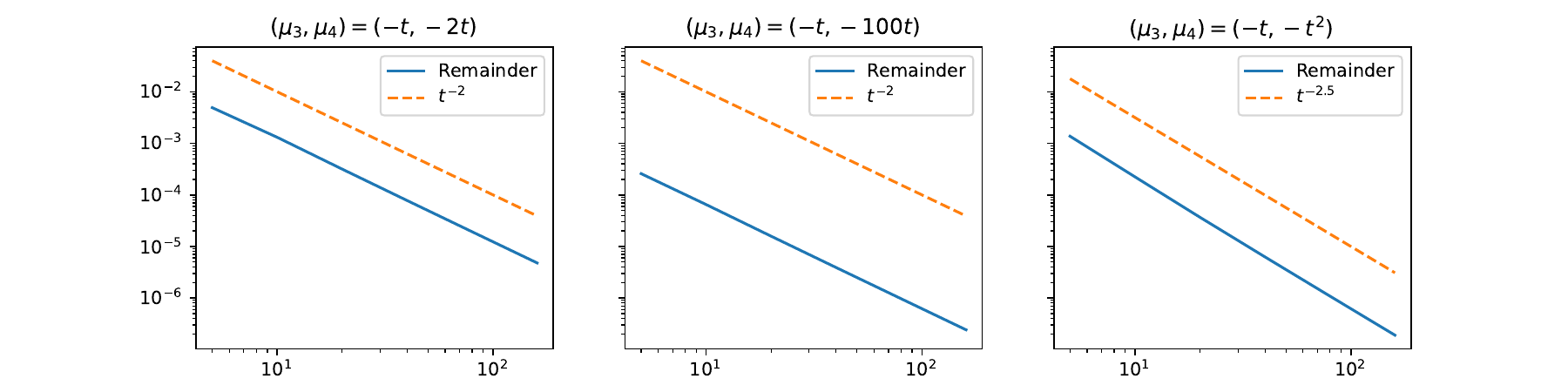}
\caption{Decay of the absolute value of $R^{(4,2)}_{0,0}=M_0^{(4)}-M_0^{(2)}/\sqrt{|\mu_3\mu_4|}$, i.e.~the asymptotic error of $M_0^{(4)}=Z^{(4)}$ with two unbounded eigenvalues $0>\mu_3>\mu_4$ and only the leading term $n=0$. We let $\mu_1=0,\mu_2=-1$, and test three different relative sizes of $(\mu_3,\mu_4)$: $(-t,-2t),(-t,-100t),(-t,-t^2)$. 
The estimate \eqref{Mdk-mult-rem} predicts that $|R^{(4,2)}_{0,0}|$ should decay at $O(|\mu_3\mu_4|^{-\frac12}(|\mu_3|^{-1}+|\mu_4|^{-1}))$, which is consistent with the slopes of the error curves ($O(t^{-2}), O(t^{-2})$ and $O(t^{-\frac52})$ in each case respectively).} \label{fig:asymp-error}
\end{figure}

As an example, we plot the asymptotic errors of \eqref{Mdk-mult-rem} in Figure \ref{fig:asymp-error}, which demonstrates the uniform error estimate independent of the relative eigenvalue sizes.
We also associate \eqref{Mdk-mult-rem} in dimension $d=2$ with the modified Bessel functions in the example below.
\begin{example}
When $d=2$ and $s=1$, we can use \eqref{Mdk-asymp-mult} to recover the well-known asymptotic expansion of the modified Bessel function $\mathrm{I}_0$. Let $\mu=(0,-2\lambda)$. We recall from the proof of Theorem \ref{thm:DZ-keyhole-int} that
\[M^{(2)}_0 = \frac{1}{\Gamma(1)} Z^{(2)}_0 (0,-2\lambda) = \e^{-\lambda} \mathrm{I}_0 (\lambda).\]
Also, when $d=1$, the quantities $M^{(d)}_k$ can be computed explicitly with Hankel's formula \eqref{gamma-hankel}:
\[M^{(1)}_k = \frac{1}{2\pi\ii}\oint_{C} z^{k-\frac12}\e^z \d z = \frac{1}{\Gamma(\frac12-k)} = \frac{(-1)^k (2k-1)!! }{2^k \sqrt{\pi}}.\]
The series \eqref{Mdk-asymp-mult} is now a univariate series with respect to $\frac{1}{-\mu_2}=\frac{1}{2\lambda}.$ With $d=2,k=0$, we get
\begin{align*}
    \e^{-\lambda}\mathrm{I}_0(\lambda) &\sim \frac{1}{\sqrt{2\lambda}}\sum_l \frac{(-1)^l (2l-1)!!}{(2l)!!} 
    \cdot \frac{(-1)^l (2l-1)!! }{2^l \sqrt{\pi}} \frac{1}{(2\lambda)^l} \\
    &= \frac{1}{\sqrt{2\pi\lambda}} \sum_l \frac{((2l-1)!!)^2}{l! \cdot (8\lambda)^{l}}.
\end{align*}
This is exactly the formula presented in \cite{abramowitz_handbook_2013}. Our result also provides the error estimate
\begin{equation} \label{I0-asymp-err}
\begin{aligned}
    &\ \left|\e^{-\lambda} \mathrm{I}_0(\lambda) - \frac{1}{\sqrt{2\pi\lambda}} \sum_{l=0}^N \frac{((2l-1)!!)^2}{l! \cdot (8\lambda)^{l}}\right| = |R^{(2,1)}_{0,N}(0,-2\lambda)| \\
    &\le \frac{6\sqrt{2}\e}{\pi \sqrt{2\lambda}} \frac{2^{N+\frac12} (N+1)!}{(2\lambda)^{N+1} \sqrt{N+2}}
    = \frac{3\sqrt{2}\e}{\pi} \frac{(N+1)!}{\lambda^{N+\frac32}\sqrt{N+2}}
\end{aligned}
\end{equation}
\end{example}

The derivatives of $Z$ also have integral representations \eqref{DZd-hankel}, so their asymptotic expansion with respect to $\nu_j\to-\infty$ can be derived likewise.
In fact, the series we obtain happens to be the same as the term-wise differentiation of \eqref{Mdk-asymp-mult} with respect to $\mu$.
\begin{corollary} \label{cor:DMdk-asymp-mult}
Work under the same assumptions as Theorem \ref{thm:Mdk-asymp-mult}. For any multi-index $\alpha\in\BbbN^d$, we denote by
\[\alpha'=(\alpha_1,\ldots,\alpha_{d-s})\in\BbbN^{d-s},\
\varpi=(\varpi_1,\ldots,\varpi_s)=(\alpha_{d-s+1},\ldots,\alpha_d)\in\BbbN^s\]
the partition of $\alpha$ corresponding to $\mu=(\mu',\nu)$.
Then, the derivative $\D^\alpha M^{(d)}_k$ has the asymptotic expansion
\begin{equation}\label{DMdk-asymp-mult}
    \D^\alpha M^{(d)}_k(\mu) = \llbracket\nu \rrbracket_{\varpi+\frac12} \sum_{|\beta|\le n} \frac{(-1)^{|\beta|}(2\beta+2\varpi-1)!!}{2^{|\beta|+|\varpi|}\beta!} \frac{\D^{\alpha'} M^{(d-s)}_{k+|\beta|}(\mu')}{(-\nu)^{\beta}} + \D^\alpha R^{(d,s)}_{k,n}(\mu),
\end{equation}
where $\varpi+\frac12\triangleq(\varpi_1+\frac12,\ldots,\varpi_s+\frac12)$ follows \eqref{multi-index}.
The remainder \eqref{Mdk-mult-rem} satisfies that
\begin{equation}\label{DMdk-asymp-mult-rem}
    \left| \D^\alpha R^{(d,s)}_{k,n}(\mu) \right| \le K_3(n,k,d,s,\alpha) 
    \llbracket\mu\rrbracket_{\alpha+\frac12} \sum_{|\beta|=n+1} (-\nu)^{-\beta}
\end{equation}
uniformly for $\nu_1\le -1$, where the constant $K_3$ has an expression
\begin{equation} \label{K3-const}
    K_3=\frac{6\sqrt{2} \e (2\alpha-1)!!}{2^{\frac{|\alpha|}{2}} \pi} K_1\qty(\varpi+\tfrac12) (n+2)^{|\varpi|+\frac{s}{2}-1} 2^{n+\frac{k+|\varpi|}{2}+\frac{d}{4}} (k+n+1)!.
\end{equation}
\end{corollary}

\begin{proof}
Differentiation under the integral sign of \eqref{Mdk} is permitted due to exponential convergence, so we have
\begin{equation} \label{DMdk}
    \D^\alpha M^{(d)}_k(\mu) = \frac{(2\alpha-1)!!}{2^{|\alpha|}} \frac{1}{2\pi\ii} \oint_{C_1} z^k \e^z \prod_{j=1}^d  (z-\mu_j)^{-\frac12-\alpha_j} \d z.
\end{equation}
The integrand has a similar form, so the rest of the derivation is identical to Theorem \ref{thm:Mdk-asymp-mult}. We only briefly list the outline.

First, we extract the factors
\[\prod_{j=1}^s (-\nu_j)^{-\frac12-\varpi_j} = (-\nu)^{-\varpi-\frac12} = \llbracket \nu \rrbracket_{\varpi+\frac12},\]
from the product, and expand $\prod_{j=1}^s (1+\frac{z}{-\nu_j})^{-\frac12-\varpi_j}$ with respect to $z$. The result is composed of the asymptotic series in \eqref{DMdk-asymp-mult}, and a remainder of the form
\begin{equation} \label{rem-adskn-expr}
    \D^{\alpha}R^{(d,s)}_{k,n}(\mu)
    = \frac{(2\alpha-1)!!}{2^{|\alpha|}}
    \frac{\llbracket \nu \rrbracket_{\varpi+\frac12}}{2\pi\ii}
    \oint_{C_1} z^k\e^z \prod_{j=1}^{d-s}(z-\mu_j)^{-\frac12-\alpha_j} r_{\varpi+\frac12;n}(z;-\nu)\d z.
\end{equation}
We note that \eqref{rem-adskn-expr} is indeed the derivative of $R^{(d,s)}_{k,n}$ \eqref{rem-dskn-expr}, since the series \eqref{DMdk-asymp-mult} happens to be the term-wise differentiation of \eqref{Mdk-asymp-mult}.

Then, we estimate \eqref{rem-adskn-expr} with similar methods as before.
\begin{align*}
    &\ \left| \D^{\alpha}R^{(d,s)}_{k,n}(\mu) \right| \\
    &\le \underbrace{\frac{(2\alpha-1)!! \llbracket \nu \rrbracket_{\varpi+\frac12} }{2^{|\alpha|}} }_{\text{Coefficient of \eqref{rem-adskn-expr}}}
    \cdot \underbrace{ \frac{\llbracket\mu'\rrbracket_{\alpha'+\frac12}}{2\pi(\sin\theta)^{\frac{d-s}{2} +|\alpha'| }} }_{\text{Same as \eqref{sqrt-dist-on-C1}}} \\
    &\hspace{40pt}
    \cdot \underbrace{ \frac{K_1(\varpi+\frac12)(n+2)^{|\varpi|+\frac{s}{2}-1}}{(\sin\theta)^{n+|\varpi|+\frac{s}{2}+1}} \sum_{|\beta|=n+1} \frac{1}{(-\nu)^\beta} }_{\text{Same as \eqref{pow-rem-on-C1}}}
    \cdot \oint_{C_1} |z|^{k+n+1}|\e^z| |{\d z}| \\
    &\le \frac{6\e (2\alpha-1)!! K_1(\varpi+\frac12)}{2\pi\cdot 2^{|\alpha|} (\sin\theta)^{|\alpha|+1} \cos^2\theta} 
    \frac{(n+2)^{|\varpi|+\frac{s}{2}-1} (n+k+1)!}{(\sin\theta)^{n+\frac{d}{2}}(\cos\theta)^{k+n}}
    \cdot \llbracket\mu\rrbracket_\alpha
    \sum_{|\beta|=n+1} \frac{1}{(-\nu)^\beta}.
\end{align*}
Letting $\theta=\frac{\pi}{4}$, we obtain the estimate \eqref{DMdk-asymp-mult-rem}, with the leading chunk of constants identified with $K_3$ \eqref{K3-const}.
\end{proof}

\subsection{Convergent limit of asymptotic series}

J.~Kent noticed in his paper \cite{kent_asymptotic_1987} that in the 3D Bingham normalizing constant with $s$ unbounded eigenvalues ($s=1,2$), the asymptotic series he obtained was absolutely convergent when $s=1$, an ``unusual property for an asymptotic series.'' We now explain and generalize this phenomenon with the integral representation \eqref{Zd-hankel}.

Recall from Corollary \ref{cor:Zd-real-int} that \emph{when there are an even number of eigenvalues, the range of integration is actually finite}. It is then natural that the expansion \eqref{pow-expand} under the integral sign should converge.
We exploit this important observation and state the following general result on $\D^\alpha M^{(d)}_k$. 

\begin{theorem} \label{thm:uni-conv-even}
Assume that $0=\mu_1\ge\cdots\ge\mu_d$, and use the same notation in Corollary \ref{cor:DMdk-asymp-mult}.
If $d-s=2t$ is even, then the degree $n$ in the series \eqref{DMdk-asymp-mult} (and likewise \eqref{Mdk-asymp-mult}) can be put to $\infty$ as long as $\nu_1<\mu_{d-s}$:
\begin{equation} \label{DMdk-asymp-infty}
    \D^\alpha M^{(d)}_k(\mu) = \frac{1}{(-\nu)^{\varpi+\frac12}} \sum_{\beta\in\BbbN^s} \frac{(-1)^{|\beta|}(2\beta+2\varpi-1)!!}{2^{|\beta|+|\varpi|}\beta!} \frac{\D^{\alpha'} M^{(d-s)}_{k+|\beta|}(\mu')}{(-\nu)^{\beta}} + \D^\alpha R^{(d,s)}_{k,\infty}(\mu).
\end{equation}
The summation on the RHS is absolutely convergent, and the remainder term satisfies that
\begin{equation} \label{Mdk-rem-infty}
    \left| \D^\alpha R^{(d,s)}_{k,\infty} \right| \le K_4(k,d,\alpha,\delta) (|\nu_1|+1)^k \e^{-|\nu_1|}
\end{equation}
uniformly for all $\nu_1\le \mu_{d-s}-2\delta$.
\end{theorem}
\begin{proof} 
We use the original Hankel-type contour $C$ with constant imaginary part (Figure \ref{fig:keyhole}) as the path of integration in \eqref{DMdk}.
The integrand is analytic on the interval $(\nu_1,\mu_{d-s})$ because there are an even number of negative square roots and the branch cuts cancel out. Hence, we can split the contour $C$ across this interval into two loops $C_{10},C_{11}$, and deform them appropriately such that $C_{10}$ closely encircles the line segment $[\mu_{d-s},\mu_1]$, and $C_{11}$, parametrized by
\begin{equation} \label{contour-C11}
C_{11}(t) = \begin{cases}
    \nu_1+\delta-\ii+(t+1), & t<-1, \\
    \nu_1+\delta+\ii t, & -1\le t \le 1, \\
    \nu_1+\delta+\ii-(t-1), & t>1,
\end{cases}
\end{equation}
is composed of two horizontal rays going to $-\infty$ and one vertical segment connecting their endpoints $\nu_1+\delta\pm\ii$ (see Figure \ref{fig:contour-split-even}).
Then, the integral \eqref{DMdk} can be written as\small
\begin{align*}
    \D^\alpha M^{(d)}_k 
    &= \frac{(2\alpha-1)!!}{2^{|\alpha|}(-\nu)^{\varpi+\frac12}} \frac{1}{2\pi\ii} \qty(\oint_{C_{10}} + \oint_{C_{11}}) z^k \e^z \prod_{j=1}^{d-s} (z-\mu_j)^{-\frac12-\alpha_j} \prod_{j=1}^s \qty(1+\frac{z}{-\nu_j})^{-\frac12-\varpi_j} \d z \\
    &\triangleq I_{10}+I_{11},
\end{align*}\normalsize
where $I_{10}$ and $I_{11}$ denote the integrals on $C_{10}$ and $C_{11}$, respectively.

\begin{figure}
\centering
\begin{tikzpicture}[decoration={markings, 
    mark=at position 0.55 with {\arrow{stealth}}}, line cap=round]
\draw[gray,->](-4.2,0)--(2,0);
\draw[gray,->](0,-1.6)--(0,1.6);
\draw[line width=2pt,opacity=0.5](-.8,0)--(0,0);
\draw[line width=2pt,opacity=0.5](-1.3,0)--(-1.5,0);
\draw[line width=2pt,opacity=0.5](-2.5,0)--(-2.9,0);
\draw[line width=2pt,opacity=0.5](-3.2,0)--(-3.8,0);
\draw[line width=2pt,opacity=0.5](-4,0)--(-4.2,0);
\fill(0,0) circle[radius=2pt];
\fill(-0.8,0) circle[radius=2pt];
\fill(-1.3,0) circle[radius=2pt];
\fill(-1.5,0) circle[radius=2pt];
\fill(-2.5,0) circle[radius=2pt];
\fill(-2.9,0) circle[radius=2pt];
\fill(-3.2,0) circle[radius=2pt];
\fill(-3.8,0) circle[radius=2pt];
\fill(-3.8,0) circle[radius=2pt];
\fill(-4,0) circle[radius=2pt];
\draw[blue,thick,postaction=decorate] (-1.5,-.2)--(0,-.2);
\draw[blue,thick,postaction=decorate] (0,-.2) arc[radius=0.2,start angle =-90,end angle=90];
\draw[blue,thick,postaction=decorate] (0,.2)node[anchor=south east]{$C_{10}$}--(-1.5,.2);
\draw[blue,thick,postaction=decorate] (-1.5,.2) arc[radius=0.2,start angle =90,end angle=270];
\draw[blue,thick,postaction=decorate] (-4.3,-1)--(-2.3,-1);
\draw[blue,thick,postaction=decorate] (-2.3,-1)--(-2.3,1);
\draw[blue,thick,postaction=decorate] (-2.3,1)node[anchor=north east]{$C_{11}$}--(-4.3,1);
\draw[red,thick,postaction=decorate] (-4,-1.2)--(0,-1.2);
\draw[red,thick,postaction=decorate] (0,-1.2) arc[radius=1.2,start angle =-90,end angle=90];
\draw[red,thick,postaction=decorate] (0,1.2)--(-4,1.2);
\draw[red] (1.2,0) node[anchor=south west]{$C$};
\draw[thick,dashed] (-1.5,-1.2)--(-1.5,-.2);
\draw[thick,dashed] (-1.5,1.2)--(-1.5,.2);
\draw[thick,dashed] (-2.5,-1.2)--(-2.5,-1);
\draw[thick,dashed] (-2.5,1.2)--(-2.5,1);
\end{tikzpicture}
\caption{Splitting the contour $C$ when there are an even number of bounded eigenvalues (4 as shown)} \label{fig:contour-split-even}
\end{figure}
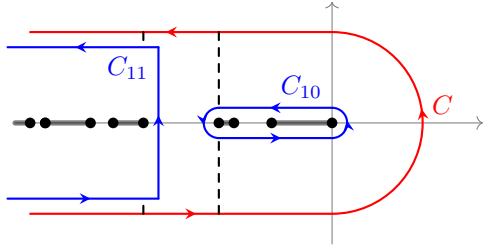

To compute $I_{10}$, we follow the previous approaches to expand the integrand in terms of $z$. 
When $\nu_1<\mu_{d-s}-2\delta$, we can shrink $C_{10}$ small enough so that
\begin{equation}
    |z|<|\nu_j|, \ \forall z\in C_{10},j=1,\ldots,s.
\end{equation}
Thus, the Taylor series of $\prod_{j=1}^s (1+\frac{z}{-\nu_j})^{-\frac12-\varpi_j}$ is uniformly absolutely convergent over $C_{10}$, and we have that\small
\begin{align*}
    &\ I_{10} = \frac{(2\alpha-1)!!}{2^{|\alpha|} (-\nu)^{\varpi+\frac12}} \sum_{\beta\in\BbbN^s} \frac{(-1)^{|\beta|}(2\beta+2\varpi-1)!!}{2^{|\beta|}\beta! (2\varpi-1)!!(-\nu)^\beta} \frac{1}{2\pi\ii}\int_{C_{10}} z^{k+|\beta|}\e^z \prod_{j=1}^{d-s}(z-\mu_j)^{-\frac12-\alpha_j}\d z \\
    &=\frac{1}{(-\nu)^{\varpi+\frac12}}
    \sum_{\beta\in\BbbN^s} \frac{(-1)^{|\beta|}(2\beta+2\varpi-1)!!}{2^{|\beta|+|\varpi|}\beta! (-\nu)^\beta}
    \frac{1}{2\pi\ii}\int_{C_{10}} z^{k+|\beta|}\e^z \prod_{j=1}^{d-s}\frac{(2\alpha_j-1)!!}{2^{\alpha_j}}(z-\mu_j)^{-\frac12-\alpha_j}\d z,
\end{align*}\normalsize
where the leading coefficient $\frac{(2\alpha-1)!!}{2^{|\alpha|}}$ is factorized into $\frac{(2\varpi-1)!!}{2^{|\varpi|}} \cdot \frac{(2\alpha'-1)!!}{2^{|\alpha'|}}$, and moved into each summand.
The integral expressions in the series above only differ from $\D^{\alpha'} M^{(d-s)}_{k+|\beta|}(\mu')$ \eqref{DMdk} by the paths of integration.
However, since $d-s$ is even, $\prod_{j=1}^{d-s}(z-\mu_j)^{-\frac12}$ is analytic on $(-\infty,\mu_{d-s})$ and the integrand of \eqref{DMdk} is exponentially convergent at $-\infty$, we can deform the Hankel-type contour $C_1$ into $C_{10}$ by closing it somewhere on $(-\infty,\mu_{d-s})$ without changing the value of the integral (similar to Figure \ref{fig:closed-contour}). This indicates that the integral expressions of the summands of $I_{10}$ are actually identical to $\D^{\alpha'} M^{(d-s)}_{k+|\beta|}(\mu')$, giving us the series
\begin{equation}
    I_{10} = \frac{1}{(-\nu)^{\varpi+\frac12}}\sum_{\beta\in\BbbN^s} \frac{(-1)^{|\beta|}(2\beta+2\varpi-1)!!}{2^{|\beta|+|\varpi|}\beta! } \frac{\D^{\alpha'} M^{(d-s)}_{k+|\beta|}(\mu')}{(-\nu)^\beta}.
\end{equation}
This indicates that the absolutely convergent limit of the sum in \eqref{DMdk-asymp-mult} as $n\to\infty$ happens to be $I_{10}$.

Consequently, the remainder $\D^\alpha R^{(d,s)}_{k,\infty}$ should be identified with $I_{11}$. 
We have that
\begin{align*}
    |I_{11}| &\le \frac{(2\alpha-1)!!}{2\pi\cdot 2^{|\alpha|}} \left(\int_{-\infty}^{\nu_1+\delta-\ii} + \int_{\nu_1+\delta-\ii}^{\nu_1+\delta+\ii} + \int_{\nu_1+\delta+\ii}^{-\infty}\right) |z|^k |\e^z| \prod_{j=1}^d|z-\mu_j|^{-\frac12-\alpha_j} |{\d z}| \\
    &\triangleq \frac{(2\alpha-1)!!}{2\pi\cdot 2^{|\alpha|}} (J_1+J_2+J_3).
\end{align*}
where the first and third integral paths refer to the rays going to $-\infty$, and the second refers to the vertical line segment. On the one hand, we bound $J_1$ by
\begin{align*}
    |J_1| &\le \int_0^\infty \qty(|\nu_1+\delta|+1+t)^k \e^{\nu_1+\delta-t} \d t \\
    &= \e^{-|\nu_1|+\delta} \sum_{j=0}^k \binom{k}{j} (|\nu_1+\delta|+1)^j \int_0^\infty t^{k-j} \e^{-t}\d t \\
    &=k! \e^{-|\nu_1|+\delta} \sum_{j=0}^k \frac{(|\nu_1+\delta|+1)^j}{j!} \\
    &\le k! \e^{1+\delta} (|\nu_1|+1)^k \cdot \e^{-|\nu_1|},
\end{align*}
which also applies to $|J_3|$ by symmetry.
On the other hand, $J_2$ is bounded by
\[|J_2| \le 2 \cdot \frac{\e^{-|\nu_1|+\delta} (|\nu_1+\delta|+1)^k}{\delta^{\frac{d}{2}+|\alpha|}}
= \frac{2\e^\delta (|\nu_1|+1)^k}{\delta^{\frac{d}{2}+|\alpha|}} \e^{-|\nu_1|}.\]
since $|z|\le |\nu_1|+1, |\e^z|\le \e^{\nu_1+\delta},$ and $|z-\mu_j|\ge|z-\nu_1|\ge \delta$ on the vertical line. 
Combining the estimates above, we finally obtain
\begin{equation}
    |I_{11}| \le \frac{(2\alpha-1)!!}{2\pi\cdot 2^{|\alpha|}} \qty(2k!\e^{1+\delta} + \frac{2\e^\delta}{\delta^{\frac{d}{2}+|\alpha|}}) (|\nu_1|+1)^k \e^{-|\nu_1|},
\end{equation}
which shrinks to 0 exponentially with respect to $\nu_1\to-\infty$. 
One obtains \eqref{Mdk-rem-infty} by setting the coefficients to $K_4(k,d,\alpha,\delta)$.
\end{proof}
\begin{remark}
It should be noted that even though the asymptotic series \eqref{DMdk-asymp-infty} is uniformly and absolutely convergent, it does not converge to the LHS in the presence of the remainder term, which is a non-trivial contour integral over $C_{11}$. 
This fact can be compared with the ``flat function'' $f(x)=\e^{-x^{-2}}$, whose Taylor expansion at $x=0$ converges, but the limit does not equal $f(x)$ itself.
\end{remark}

\section{Entropy decomposition in Bingham closure}

In this section, we apply the asymptotic theory to the Bingham closure problem. We first introduce the Bingham moments and their asymptotic estimates. After that, we study the Bingham closure and the expression of the entropy \eqref{bingham-ent} in terms of second-order moments. Finally, we prove the main result of entropy decomposition, which generalizes a previous result by the authors \cite{shi_molecular_2026} into arbitrary dimensions.

\subsection{Bingham moments}

Consider the Bingham distribution $\mathrm{Bing}(\mu)$ (\eqref{bingham} with diagonal parameter $B=\diag(\mu)$). Since $f(m;\mu)$ is antipodally symmetric, only moments of even orders are non-trivial. Recalling the expression for $\D^\alpha Z$ \eqref{DZd-mu}, we define the \textbf{Bingham moments} as
\begin{equation} \label{bing-moments}
    z_\alpha = \BbbE^\mu(m^{2\alpha}) = \frac{\D^\alpha Z}{Z},
\end{equation}
where $\alpha\in\BbbN^d$ is a multi-index, and $\BbbE^\mu$ denotes the expectation under $\mathrm{Bing}(\mu)$. Evidently, all moments are strictly positive.
For lower-order $\alpha$, we also use the following abbreviations:
\begin{align*}
    z_j &= z_{e_j} = \frac{1}{\omega_d Z} \int_{\BbbS^{d-1}} m_j^2 \e^{\sum_j \mu_j m_j^2}\d m, \\
    z_{jk}&=z_{e_j+e_k} = \frac{1}{\omega_d Z} \int_{\BbbS^{d-1}} m_j^2 m_k^2 \e^{\sum_j \mu_j m_j^2}\d m.
\end{align*}
We also denote by
\begin{equation}
    z=[z_1,\ldots,z_d]^T
\end{equation}
the vector formed by the second-order moments, which, since $\sum_j m_j^2=1$, satisfies the constraint that
\begin{equation} \label{sum-zj-is-1}
    \sum_{j=1}^d z_j = 1.
\end{equation}

The moments are central in the study of Bingham distributions \cite{ball_nematic_2010,chen_maximum_2021,de_gennes_physics_1993,kirschvink_least-squares_1980}, where the common approach is to collect moments from experimental observations, and then resolve the underlying parameter $\mu$ by an inversion known as the Bingham closure.

\subsection{Asymptotic profile of Bingham moments}

First, we need to study the asymptotic properties of the Bingham moments, where the special case in three dimensions was analyzed by Kent \cite{kent_asymptotic_1987}, and we generalize it to arbitrary dimensions.

To distinguish between dimensions, we restore the superscript $(d)$ for the function $Z$ and its relatives in dimension $d$.
We begin with a crude estimate that restricts the decay rate of $\D^\alpha Z^{(d)}$ to the same order as the modulus $\llbracket \mu \rrbracket_{\alpha+\frac12}$.
In particular, the bounds here are \emph{independent of the number of unbounded eigenvalues,} as distinct from the asymptotic series \eqref{DMdk-asymp-mult} and \eqref{Mdk-asymp-mult}.

\begin{lemma} \label{lem:DZ-order}
Suppose that the parameters are nonpositive and nonincreasing:
\[0=\mu_1\ge\cdots\ge \mu_d.\]
For any $\alpha\in\BbbN^d$, there exist positive constants $K_5^+>K_5^->0$ (dependent on $d,\alpha$) such that
\begin{equation} \label{DZ-order}
    K_5^-\le \frac{ \D^\alpha Z^{(d)}(\mu)}{\llbracket \mu \rrbracket_{\alpha+\frac12} } \le K_5^+,
\end{equation}
with the modulus $\llbracket \cdot \rrbracket$ defined in \eqref{w-alpha}.
\end{lemma}
\begin{proof}
We induct on the dimensionality $d$. The case in $d=1$ is trivial as $Z^{(1)}(\mu_1)=\e^{\mu_1}=1$.
Supposing that the result is correct for $d-1$, we prove \eqref{DZ-order} for $d\ge 2.$ The general idea is to split all $\mu$'s into a compact set and its complement, and apply asymptotic expansions to the latter.

First, if $\mu_d\ge -X$ for some $X>1$, then by the explicit expression \eqref{DZd-mu} we have that
\[\D^\alpha Z^{(d)}(\mu)\ge c_X >0,\]
because $m^{2\alpha}$ is nonzero almost everywhere on $\BbbS^{d-1}$.
As $\llbracket \mu \rrbracket_{\alpha+\frac12}$ is also a strictly positive continuous function, there exist constants $c_1(X,d,\alpha)$ and $c_2(X,d,\alpha)$ such that 
\begin{equation} \label{DZ-order-bnd}
    c_2(X,d,\alpha) \le \frac{\D^\alpha Z^{(d)}(\mu)}{\llbracket \mu \rrbracket_{\alpha+\frac12}} \le c_1(X,d,\alpha),
\end{equation}
provided that $0\ge\mu_1\ge\cdots\ge\mu_d\ge -X.$

Then, if $\mu_d<-X<-1$, we use the asymptotic expansion \eqref{DMdk-asymp-mult} with respect to $\mu_d$ at order $0$ to get
\begin{equation} \label{DZ-asymp-0}
    \D^\alpha Z^{(d)}(\mu) = \frac{\Gamma(\frac{d}{2})}{\Gamma(\frac{d-1}{2})}\frac{(2\alpha_d-1)!!}{2^{\alpha_d}} \llbracket \mu_d \rrbracket_{\alpha_d+\frac12} \D^{\alpha'} Z^{(d-1)}(\mu') + \Gamma(\tfrac{d}{2}) \D^\alpha R^{(d,1)}_{0,0}(\mu),
\end{equation}
where $\mu'=[\mu_1,\ldots,\mu_{d-1}]^T, \alpha'=(\alpha_1,\ldots,\alpha_{d-1})$, and we used the relation $Z^{(d)}=\Gamma(\frac{d}{2})M_0^{(d)}$ \eqref{Zd-is-Md0} to translate between $M_0$ and $Z$. We also have the error estimate \eqref{DMdk-asymp-mult-rem}:
\begin{equation} \label{DZ-asymp-0-rem}
    \left|\Gamma(\tfrac{d}{2})\D^\alpha R^{(d,1)}_{0,0}\right| \le \Gamma(\tfrac{d}{2}) K_3(0,0,d,1,\alpha) \frac{\llbracket\mu \rrbracket_{\alpha+\frac12}}{|\mu_d|} \le c_3(d,\alpha) \frac{\llbracket\mu \rrbracket_{\alpha+\frac12}}{X}.
\end{equation}
Note carefully that the remainder estimate \eqref{DMdk-asymp-mult-rem} is \emph{a uniform bound provided $\mu_d\le -1$}, so the coefficient $c_3$ in \eqref{DZ-asymp-0-rem} does not depend on the remaining eigenvalues $\mu'$ even if they are unbounded.

Using the induction hypothesis on $\D^{\alpha'} Z^{(d-1)}(\mu')$, we find that
\begin{align*}
    \D^\alpha Z^{(d)} &\le \frac{\Gamma(\frac{d}{2})}{\Gamma(\frac{d-1}{2})} \frac{(2\alpha_d-1)!!}{2^{|\alpha_d|}} \llbracket \mu_d \rrbracket_{\alpha_d+\frac12} \cdot K_5^+(d-1,\alpha') \llbracket\mu' \rrbracket_{\alpha'+\frac12} + c_3(d,\alpha) \frac{\llbracket\mu \rrbracket_{\alpha+\frac12}}{X}, \\
    \D^\alpha Z^{(d)} &\ge \frac{\Gamma(\frac{d}{2})}{\Gamma(\frac{d-1}{2})} \frac{(2\alpha_d-1)!!}{2^{|\alpha_d|}} \llbracket \mu_d \rrbracket_{\alpha_d+\frac12} \cdot K_5^-(d-1,\alpha') \llbracket\mu' \rrbracket_{\alpha'+\frac12} - c_3(d,\alpha) \frac{\llbracket\mu \rrbracket_{\alpha+\frac12}}{X}.
\end{align*}
We can extract the factors $\llbracket\mu \rrbracket_\alpha= \llbracket\mu' \rrbracket_{\alpha'+\frac12}\llbracket \mu_d \rrbracket_{\alpha_d+\frac12}$ (by the factorization property \eqref{w-alpha-div}), and obtain
\begin{subequations} \label{DZ-order-ubnd}
\begin{equation}
    c_5(d,\alpha) \llbracket\mu \rrbracket_{\alpha+\frac12} \le \D^\alpha Z^{(d)} \le c_4(d,\alpha) \llbracket\mu \rrbracket_{\alpha+\frac12}, 
\end{equation}
where
\begin{equation}
\begin{aligned}
    c_4 &= \frac{\Gamma(\frac{d}{2})}{\Gamma(\frac{d-1}{2})}\frac{(2\alpha_d-1)!! K_5^+(d-1,\alpha')}{2^{|\alpha_d|}} + \frac{c_3(d,\alpha)}{X}, \\
    c_5 &= \frac{\Gamma(\frac{d}{2})}{\Gamma(\frac{d-1}{2})}\frac{(2\alpha_d-1)!! K_5^-(d-1,\alpha')}{2^{|\alpha_d|}} - \frac{c_3(d,\alpha)}{X}.
\end{aligned}
\end{equation}
\end{subequations}
Therefore, we should choose $X>\frac{\Gamma(\frac{d-1}{2})}{\Gamma(\frac{d}{2})}\frac{2^{|\alpha_d|} c_3(d,\alpha)}{(2\alpha_d-1)!! K_5^-(d-1,\alpha')}$ at the beginning to make $c_5$ positive.

Finally, combining the estimates \eqref{DZ-order-bnd} and \eqref{DZ-order-ubnd} gives us \eqref{DZ-order}, with
\[K_5^+(d,\alpha)=\max\{c_1,c_4\},\ 
K_5^-(d,\alpha)=\min\{c_2,c_5\}.\]
\end{proof}

Lemma \ref{lem:DZ-order} immediately yields lower and upper bounds on $z_\alpha^{(d)}$: by dividing the estimates \eqref{DZ-order} for orders $\alpha$ and $0$, we get
\begin{equation} \label{za-order}
    \frac{K_5^-(d,\alpha)}{K_5^+(d,0)} \le \frac{z_\alpha^{(d)}}{\llbracket \mu \rrbracket_\alpha} \le \frac{K_5^+(d,\alpha)}{K_5^-(d,0)},
\end{equation}
so $z_\alpha^{(d)}$ decays at the same rate as $\llbracket \mu \rrbracket_\alpha=\prod_j \max\{1,|\mu_j\}^{-\alpha_j}.$
We will abbreviate this relation by $z^{(d)}_\alpha \asymp \llbracket \mu \rrbracket_\alpha$ from now on.

We can also provide a more accurate description of the asymptotic leading term by dividing $\D^\alpha Z^{(d)}$ with $Z^{(d)}$.

\begin{proposition} \label{prop:za-asymp}
Similar to Corollary \ref{cor:DMdk-asymp-mult}, we assume that $0=\mu_1\ge\cdots\ge\mu_d$, and denote by
\[\mu=\mqty[\mu'\\\nu],\ \alpha=(\alpha',\varpi)\]
the partition of variables into dimensions $d-s$ and $s$. Then, the moment $z_\alpha^{(d)}$ satisfies that
\begin{equation} \label{za-asymp}
    z_\alpha^{(d)} = \frac{(2\varpi-1)!!}{2^{|\varpi|}}\llbracket \nu \rrbracket_\varpi z_{\alpha'}^{(d-s)}(\mu') + \ve^{(d)}_\alpha(\mu)
\end{equation}
with a uniform error estimate
\begin{equation} \label{za-asymp-err}
    \left|\ve^{(d)}_\alpha\right| \le K_6(d,s,\alpha) \llbracket \mu \rrbracket_\alpha \sum_{j=1}^s \frac{1}{|\nu_j|}
\end{equation}
for $\nu_1\le -1$.
\end{proposition}
\begin{proof}
We use the multivariate expansion \eqref{DMdk-asymp-mult} on $Z^{(d)}$ and $\D^\alpha Z^{(d)}$ at order $0$ (using \eqref{Zd-is-Md0} to convert $M^{(d)}_0$ and $M^{(d-s)}_0$):
\begin{align*}
    Z^{(d)} &= \frac{\Gamma(\tfrac{d}{2})}{\Gamma(\tfrac{d-s}{2})}\llbracket \nu \rrbracket_{\frac12} Z^{(d-s)}(\mu') + \Gamma(\tfrac{d}{2}) R^{(d,s)}_{0,0}(\mu), \\
    \D^\alpha Z^{(d)} &= \frac{\Gamma(\tfrac{d}{2})}{\Gamma(\tfrac{d-s}{2})} \frac{(2\varpi-1)!!}{2^{|\varpi|}}\llbracket \nu \rrbracket_{\varpi+\frac12} \D^{\alpha'} Z^{(d-s)}(\mu') +  \Gamma(\tfrac{d}{2}) \D^\alpha R^{(d,s)}_{0,0}(\mu).
\end{align*}
Dividing them and extracting the leading term, we get 
\begin{align*}
    z^{(d)}_\alpha &= \frac{\D^\alpha Z^{(d)}}{Z^{(d)}} \\
    &=\frac{(2\varpi-1)!!}{2^{|\varpi|}}\llbracket \nu \rrbracket_{\varpi} z^{(d-s)}_{\alpha'}(\mu') \\
    &\quad +\underbrace{\Gamma(\tfrac{d}{2}) \frac{\D^\alpha R^{(d,s)}_{0,0}(\mu) - \frac{(2\varpi-1)!!}{2^{|\varpi|}}\llbracket\nu\rrbracket_\varpi z^{(d-s)}_{\alpha'}(\mu') R^{(d,s)}_{0,0}(\mu)}{Z^{(d)}(\mu)}}_{\ve_\alpha^{(d)}(\mu)}.
\end{align*}
The error $\ve_\alpha^{(d)}(\mu)$ has been identified with the fraction.

The following bounds come from our previous estimates:
\begin{align*}
    \left| \D^\alpha R^{(d,s)}_{0,0}(\mu) \right| &\le K_3(0,0,d,s,\alpha)\llbracket\mu\rrbracket_{\alpha+\frac12} \sum_{j=1}^s \frac{1}{|\nu_j|}, \\
    \left| R^{(d,s)}_{0,0}(\mu)\right| &\le K_3(0,0,d,s,0)\llbracket\mu\rrbracket_{\frac12} \sum_{j=1}^s \frac{1}{|\nu_j|}, \text{ by \eqref{DMdk-asymp-mult-rem};}\\
    z^{(d-s)}_{\alpha'}(\mu') &\le \frac{K_5^+(d-s,\alpha')}{K_5^-(d-s,0)} \llbracket\mu'\rrbracket_{\alpha'}, \text{ by \eqref{za-order};}\\
    Z^{(d)} &\ge K_5^-(d,0) \llbracket\mu\rrbracket_{\frac12}, \text{ by \eqref{DZ-order}.}
\end{align*}
Note that a lower bound is given for the denominator, while upper bounds are given for the numerator.
Combining these bounds and choosing the largest among the constants, we get
\begin{align*}
    |\ve_\alpha^{(d)}(\mu)| \le K_6(d,s,\alpha) \frac{\llbracket\mu\rrbracket_{\alpha+\frac12}}{\llbracket\mu\rrbracket_{\frac12}} \sum_{j=1}^s  \frac{1}{|\nu_j|} = K_6(d,s,\alpha) \llbracket\mu\rrbracket_{\alpha}  \sum_{j=1}^s \frac{1}{|\nu_j|},
\end{align*}
matching \eqref{za-asymp-err} as desired.
\end{proof}

\begin{example}
Under the notation of Proposition \ref{prop:za-asymp}, we apply the asymptotic estimate \eqref{za-asymp} to the lowest-order moments $z_j$ by setting $\alpha=e_j$ the unit vector. As $\nu_1\to-\infty$, if $j\le d-s$ is associated with $\mu'$, then $\varpi=0$ and $\alpha'$ is a unit vector in $d-s$ dimensions, so we have that
\begin{subequations} \label{zj-asymp}
\begin{equation}
    z_j^{(d)}(\mu) = z_j^{(d-s)}(\mu') + O\qty(\llbracket \mu_j\rrbracket_1 \sum_{k=1}^s \frac{1}{|\nu_k|}).
\end{equation}
That is, the second-order moments corresponding to the bounded $\mu_j$ converge to their $(d-s)$-dimensional counterparts.
Otherwise, if $j>d-s$, then $\alpha'=0$ and $\varpi$ is a unit vector with $|\varpi|=1, (2\varpi-1)!!=1$, and
\begin{equation}
    z_j^{(d)}(\mu) = \frac{1}{2|\mu_j|} + O\qty(\llbracket \mu_j\rrbracket_1 \sum_{k=1}^s \frac{1}{|\nu_k|}).
\end{equation}
\end{subequations}
That is, moments corresponding to unbounded $\mu_j$ converge to 0 at the precise order of $O(|\mu_j|^{-1})$, which is exactly the asymptotic profile derived by Kent \cite{kent_asymptotic_1987}. 
These asymptotic relations are very useful as initial guesses in Bingham closure solvers.
\end{example}



\subsection{Bingham closure and entropy function}

We remove the superscript $(d)$ again in this subsection, and now proceed to the main topic of this section: the Bingham closure.

It is conventional to normalize the Bingham parameter $B=\diag(\mu)$ by the traceless gauge condition $\tr B=0$. In terms of eigenvalues, $\mu$ is restricted to the set
\begin{equation}
    S_0^d \triangleq \{x\in\BbbR^d: x_1+\cdots+x_d=0\}.
\end{equation}
Evidently, $S_0^d$ is the orthogonal complement of $\vec 1=[1,1,\ldots,1]^T$ in $\BbbR^d$, and serves as a canonical representative space for the quotient space over the equivalence classes $\{\mu+t\}$ of Bingham distribution parameters. We denote by
\begin{equation} \label{proj-grad}
    \nabla_0 F(\mu) = \qty(I - \frac{1}{d}\vec 1 \vec 1^T)\nabla F(\mu)
\end{equation}
the orthogonal projection of the regular gradient in $\BbbR^d$ onto $S_0^d$, which can also be viewed as the gradient operator on $S_0^d$.

As has been shown in \eqref{bingham-mmnt}, the second-order moments $z_j$ arise as the gradient over $\mu$. Since $\mu\in S_0^d$, the gradient should also be projected to $S_0^d$. By the constraint \eqref{sum-zj-is-1}, the moment vector $z$ stays in a hyperplane perpendicular to $\vec 1$ and parallel to $S_0^d$, so it should be shifted by a fixed length along the direction $\vec 1$:
\begin{equation} \label{qi-def}
    q(\mu)=z(\mu)-\frac1d \in S_0^d.
\end{equation}
(The vector-scalar-addition notation in \eqref{multi-index} is used.) When the dimension $d=3$, this is the well-known Q-tensor in diagonal form \cite{de_gennes_short_1971,mottram_introduction_2014}.
Since $0<z_j<1$ by definition, the vector $q$ is contained in the following set:
\begin{equation} \label{eigq-phy}
    \Delta_0^d=\left\{x\in S_0^d: x_j\in\left(-\frac1d,\frac{d-1}d\right)\right\}.
\end{equation}

Formally, the Bingham closure is defined as the process of determining a unique $\mu\in S_0^d$ for all $q\in \Delta_0^d$, such that $q=q(\mu)$ satisfies \eqref{qi-def}. Sometimes, the term ``Bingham closure'' is also used to refer to the resulting $\mu$ or the distribution $\mathrm{Bing}(\mu)$ as a function of $q$. 
The main result of Bingham closure is existence and uniqueness, i.e.~the mapping $q(\mu)$ is both injective and surjective from $S_0^d$ onto $\Delta_0^d$. 
First, we prove the injectivity of $q=q(\mu)$ with a strict convexity argument, which generalizes \cite{ball_nematic_2010} into arbitrary dimension.
\begin{proposition} \label{prop:Zd-conv}
The function $\ln Z(\mu)$ satisfies that:
\begin{enumerate}[\rm(i)]
\item $\ln Z(\mu)$ is smooth and strictly convex on $S_0^d$.
\item The gradient of $\ln Z$ in $S_0^d$ is $\nabla_0(\ln Z)=q$.
\item The Hessian of $\ln Z(\mu)$ on $\BbbR^d$ is the covariance matrix of $(m_1^2,\ldots,m_d^2)$ under $\mathrm{Bing}(\mu)$:
\begin{equation} \label{hess-lnZ}
    H(\mu) = ((z_{jk}-z_jz_k))_{d\times d}
\end{equation}
and it is strictly positive definite when restricted to  $S_0^d$.
\end{enumerate}
As a result, the gradient mapping $\nabla_0(\ln Z)$ is a smooth diffeomorphism from $S_0^d$ onto its image.
\end{proposition}

\begin{proof}
(i) By H\"older's inequality, for any $t\in(0,1)$ and $\mu,\mu'\in S_0^d$,
\[\int_{\BbbS^{d-1}} \e^{\sum_j (t\mu_j +(1-t)\mu'_j) m_j^2}\d m 
\le \qty(\int_{\BbbS^{d-1}} \e^{\sum_j \mu_j m_j^2}\d m)^t \qty(\int_{\BbbS^{d-1}} \e^{\sum_j \mu'_j m_j^2}\d m)^{1-t}.\]
Taking the logarithm of both sides and adding a few constants gives us
\[\ln Z(t \mu +(1-t)\mu' ) \le t \ln Z(\mu)+(1-t)\ln Z(\mu').\]
The equality holds if and only if $\e^{\sum_j \mu_j m^2}\propto \e^{\sum_j \mu_j' m_j^2}.$ Evaluating these two functions at the unit vectors $e_k$, $k=1,\ldots,d$ gives us
\[\mu_j-\mu'_j=\text{const},\]
which indicates that $\mu=\mu'$ as $\mu,\mu'\in S_0^d$.

(ii) Direct computation gives us
\[\pt_j(\ln Z) = \frac{\pt_j Z}{Z} = z_j,\]
so $\nabla(\ln Z)=z$. Then, by \eqref{proj-grad}, we have
\[\nabla_0(\ln Z) = z-\frac{\vec 1}{d}\sum_j z_j = z-\frac1d = q\]
as desired, where the constraint \eqref{sum-zj-is-1} is used.

(iii) We first derive the expression \eqref{hess-lnZ}, which is obtained by differentiating $z_j=\frac{\pt_j Z}{Z}$:
\[H_{jk}=\pt_{jk}(\ln Z)=\pt_k\qty(\frac{\pt_j Z}{Z}) = \frac{Z \pt_{jk}Z - \pt_k Z \pt_j Z}{Z^2}=z_{jk}-z_jz_k.\]
Then, for any nonzero $x\in S_0^d$, we have that
\begin{align*}
    Hx\cdot x &=
    \sum_{j,k=1}^d z_{jk} x_jx_k - \qty(\sum_{j=1}^d z_jx_j)^2 \\
    &=\BbbE^\mu\qty(\sum_{j=1}^d x_j m_j^2)^2 - \qty(\BbbE^\mu \sum_{j=1}^d x_j m_j^2)^2,
\end{align*}
which is the variance of $\sum_j x_jm_j^2$ under $\mathrm{Bing}(\mu)$.
If $Hx\cdot x=0$, then $\sum_j x_j m_j^2$ must stay constant on $\BbbS^{d-1}$ since the probability density $f(m)$ is nonzero everywhere. This is impossible for a nonzero $x\in S_0^d$. Therefore, $H$ is positive definite restricted to $S_0^d$.

Finally, since $\ln Z$ is strictly convex, its gradient mapping $q(\mu)=\nabla_0(\ln Z)$ must be injective on $S_0^d$; since its Hessian is positive definite everywhere, for every $q$ in the range there exists a local Bingham closure $\mu=\mu(q)$ by the inverse function theorem, and the mapping between $q$ and $\mu$ is a diffeomorphism.
\end{proof}

Then, we prove surjectivity. It suffices to show that the range of $q(\mu)$ actually covers the entire $\Delta_0^d$.
\begin{proposition} \label{prop:qj-surj}
The mapping $q=q(\mu)$ \eqref{qi-def} is surjective from $S_0^d$ onto $\Delta_0^d$.
\end{proposition}
\begin{proof}
First, we show that the range of $q=\nabla_0(\ln Z)$, which we denote by $q(S_0^d)\subset \Delta_0^d$, is a convex set. Then, we use point sequences in $q(S_0^d)$ to approximate vertices of the simplex $\Delta_0^d$, so that any given point $P\in\Delta_0^d$ is contained in their convex hull in the limiting sense, and therefore belongs to $q(S_0^d)$. 

Convexity of the gradient range is guaranteed by the following elementary lemma (a special case of \cite[Thm.~14.17]{bauschke_convex_2017} for smooth functions). The conclusion is stated in $\BbbR^n$, but by identifying $S_0^d$ with $\BbbR^{d-1}$ through an orthonormal basis, one can directly apply it to the function $\ln Z$.
Hence, $q(S_0^d)$ is convex in $S_0^d$.
\begin{lemma}\label{lem:conv-dom-conj}
Suppose that $F$ is a smooth and strictly convex function on $\BbbR^n$. Then, the range of $\nabla F$ forms a convex set in $\BbbR^n$.
\end{lemma}
\begin{proof}[Proof of Lemma \ref{lem:conv-dom-conj}]
Denote by $D=\{g: g=\nabla F(x), x\in\BbbR^n\}.$ 
We show that $g\in D$ if and only if $F(x)-g\cdot x$ is coercive, i.e.
\begin{equation} \label{F-gx-coer}
    \lim_{|x|\to\infty} [F(x)-g\cdot x]=\infty
\end{equation}
uniformly in all directions. Then, convexity of $D$ follows naturally since the points where \eqref{F-gx-coer} holds obviously form a convex set.

If \eqref{F-gx-coer} holds, then $F(x)-g\cdot x$ must attain its minimum at a finite point $x^*$. Taking the gradient gives us $\nabla F(x^*) = g$, so $g\in D.$
Conversely, suppose that $g=\nabla F(x^*)\in D$, and we aim to prove the coercivity of $F(x)-g\cdot x$. By substituting $\tilde F(y) = F(x^*+y)-F(x^*)-g\cdot y,$
which is also strictly convex, we can further assume w.l.o.g.~that $g=0$ and $\nabla F(0)=0$.
By strict convexity, $F(x)$ attains its strict minimum along any line through $x=0$, so for any unit vector $v\in \BbbS^{n-1}$ we have that $F(v)>0$. By the compactness of the unit sphere, we have that
\[\inf_{v\in \BbbS^{n-1}} F(v)= \delta>0.\]
For any $|x|>1$, denote by $\hat x = x/|x|$. We apply the convexity of $F$ to the collinear points $x,\hat x$ and 0, and get that
\begin{align*}
    F(\hat x) &< \qty(1-\frac{1}{|x|})F(0) + \frac{1}{|x|} F(x) \\
    \Rightarrow F(x) &> |x| F(\hat x) \ge \delta |x|\to \infty,
\end{align*}
so coercivity \eqref{F-gx-coer} holds.
\end{proof}

We then construct sequences that approximate the vertices $e_1-\frac1d,\cdots,e_d-\frac1d$ of $\Delta_0^d$. By the symmetry of $\ln Z$, the components of $z(\mu)=\nabla (\ln Z(\mu))$ (and hence of $q(\mu)=z(\mu)-\frac1d$) are permuted accordingly as those of $\mu$ are permuted, so it suffices to construct a sequence of $q$'s that approach $e_1-\frac1d$.
Letting $\mu(t)=[(d-1)t, -t,\ldots, -t]^T$ and $t\to\infty$, we have by the order estimate $z_j=O(\llbracket \mu_j\rrbracket_1)$ that
\[z_j(\mu) = z_j(0,-dt,\ldots,-dt) \le  \frac{C}{dt}\to 0,\ j=2,3,\ldots,d,\]
where we used the equivalence relation \eqref{Z-mupt-eq} to translate $\mu\in S_0$ to $[0,-dt,\ldots,-dt]^T$, and $C=\frac{K_5^+(d,e_j)}{K_5^-(d,0)}$ is the constant from  \eqref{za-order} independent of $t$.
By the constraint \eqref{sum-zj-is-1}, this implies that $z_1(\mu(t))\to 1$. 
Rewriting by $q=z-\frac1d$, we have that $q_1\to 1-\frac1d$ and $q_j\to -\frac1d$ for $j=2,3,\ldots,d$, so the sequence $q(\mu(t))$ approximates $e_1-\frac1d$ as desired.

Now that we have obtained sequences $\{Q^j_n\}\subset q(S_0^d)$ such that $\lim_{n\to\infty} Q^j_n =e_j-\frac1d$, for $n$ sufficiently large and any $P\in \Delta_0^d$, we can write out its barycentric coordinates with respect to $Q^1_n,\cdots, Q^d_n$ as
\begin{equation}
    P = \lambda_1 Q^1_n + \cdots + \lambda_d Q^d_n,
\end{equation}
since the points $e_j-\frac1d$ form a non-degenerate simplex in the space $S_0^d$, and $Q^1_n,\cdots, Q^d_n$ approximate them.
Denoting by $P=[p_1,\ldots,p_d]^T$, we have by basic linear algebra that
\begin{equation}
    P = \qty(I-\frac1d \vec 1 \vec 1^T)\qty(P+\frac1d \vec 1)=\qty(p_1+\frac1d)\qty(e_1-\frac1d) + \cdots + \qty(p_d+\frac1d)\qty(e_d-\frac1d).
\end{equation}
Since barycentric coordinates are continuous, their limit equals
\[\lim_{n\to\infty} \lambda_j = p_j+\frac1d>0\]
by the definition of $\Delta_0^d$ \eqref{eigq-phy}. Hence, for sufficiently large $n$, it must hold that $\lambda_j>0$ for all $j=1,\ldots, d$, and that $P$ is contained in the convex hull of $Q^1_n,\ldots, Q^d_n$. Therefore, $P \in q(S_0^d)$ as desired.
\end{proof}

Combining Propositions \ref{prop:Zd-conv} and \ref{prop:qj-surj} immediately yields the following result on Bingham closure.
\begin{theorem} \label{thm:bingcl}
    The Bingham closure $\mu=\mu(q)$ exists uniquely for any $q\in \Delta_0^d$.
\end{theorem}

\subsection{Entropy decomposition and Lipschitz regularity of residual}

With the Bingham closure established, one can now model the Bingham distribution through its second-order moments $q$ (or $z$), which are more directly obtainable in experiments compared to the parameters $\mu$. 
In particular, the entropy \eqref{bingham-ent} of the Bingham distribution can now be regarded as a function of $q$. If $\mu=\mu(q)$ is the Bingham closure of $q\in \Delta_0^d$, we compute directly that
\begin{align*}
    S &= \int_{\BbbS^{d-1}} f \ln(\frac{\exp(\sum_j \mu_j m_j^2)}{Z}) \d m =\BbbE^\mu\qty(\sum_j \mu_j m_j^2 - \ln Z) \\
    &=\sum_j \mu_j z_j - \ln Z.
\end{align*}
Since $\mu$ is traceless by the gauge condition $\mu\in S_0^d$, the expression above is equivalent to
\begin{equation}\label{S-q}
    S(q)=\mu\cdot q-\ln Z.
\end{equation}
Recall that $q$ is the gradient of $\ln Z$ over $\mu\in S_0^d$, so \eqref{S-q} is the Legendre transform of $\ln Z$ \cite{boyd_convex_2004}, and the Bingham entropy $S(q)$ is the convex dual of $\ln Z$. The function $S(q)$ is known to blow up logarithmically as $q$ approaches the boundaries of $\Delta_0^d$ when $d=3$, posing challenges in its computation.

To tackle these challenges, the authors have recently proposed a decomposition of $S(q)$ into an explicit but singular leading part and an implicit but sufficiently smooth residual part \cite{shi_molecular_2026}. In arbitrary dimensions $d$, our result is stated as follows.
\begin{theorem} \label{thm:Sq-asymp}
The entropy $S(q)$ \eqref{S-q} satisfies that
\begin{equation} \label{Sq-asymp}
    S(q) = -\frac12\ln(z_1\cdots z_d) + \Delta S(q),
\end{equation}
where $z=q+\frac1d$, and $\Delta S(q)$ is a uniformly Lipschitz function that can be extended continuously to the closure of $\Delta_0^d$.
\end{theorem}
\begin{remark}
The leading term
\begin{equation}\label{S-hat}
    \hat S(q) = -\frac12\ln(z_1\cdots z_d)
\end{equation}
captures the singular behavior of $S(q)$ near the boundary very accurately. It is reminiscent of a ``quasi-entropy'' with the same logarithmic form \cite{xu_quasi-entropy_2022}, which was intended to be a surrogate function for the implicit Bingham entropy in three dimensions:
\[\Xi(q) = -\nu \sum_j \ln(z_j(1-z_j)).\]
We will refer to the function $\hat S$ by the term ``quasi-entropy'' as well, but we note here that $\hat S$ and $\Xi$ have essentially different singular behavior near boundaries. Our quasi-entropy $\hat S$ has a Lipschitz-continuous residual term while $\Xi$ does not.
\end{remark}

Using the asymptotics established in Proposition \ref{prop:za-asymp}, we can prove this theorem in a much easier way than our previous work. Here we introduce the $O(\cdot)$ notation, but we will explicitly declare the dependence relations of the bounding constant whenever it appears.
\begin{proof}
Since $S$ is the convex dual of $\ln Z$ and $q=\nabla_0(\ln Z(\mu))$, $\nabla_0 S(q) = \mu$; moreover, 
\[\nabla_0 \hat S = \left[-\frac{1}{2z_1},\ldots,-\frac{1}{2z_d} \right]^T-c\vec 1,\]
where $c \vec 1$ acts as the projection \eqref{proj-grad}. 
Thus, it suffices to bound the pairwise differences
\begin{equation}\label{grad0-dS-bnd}
    \sup_{q\in\Delta_0^d} \left| \qty(\mu_j + \frac{1}{2z_j}) - \qty(\mu_k + \frac{1}{2z_k}) \right| < \infty,\ j,k=1,\ldots,d.
\end{equation}
The reformulation \eqref{grad0-dS-bnd} cancels out the projection constant $c$, and stays invariant under common shifts of $\mu_j$. By the symmetry of $Z$ with respect to permutations, we can therefore assume that
\[0=\mu_1\ge\cdots\ge\mu_d.\]

Using the integral representation \eqref{Zd-hankel}, it is straightforward to derive the identity
\begin{equation} \label{zj-zk-zjk-id}
    z_j-z_k = 2(\mu_j-\mu_k) z_{jk},
\end{equation}
(see the appendix for a detailed proof). Due to the positivity of moments, $z_j\ge z_k$ if and only if $\mu_j\ge\mu_k$. By the ordering of $\mu_j$, it then holds that
\[z_1\ge z_2\ge\cdots\ge z_d.\]
As $z_1+\cdots+z_d=1$, we have that $z_1\ge\frac1d$ is bounded away from zero, so $\mu_1+\frac{1}{2z_1}$ is uniformly bounded. Hence, the verification of \eqref{grad0-dS-bnd} reduces to the following bound:
\begin{equation} \label{grad0-dS-bnd-1}
    \sup_{\mu_d\le\ldots\le\mu_2\le 0} \left| \mu_j + \frac{1}{2z_j} \right| <\infty,\ j=2,3,\ldots,d.
\end{equation}
Note that the variable has been changed from $q$ to $\mu$.
We can then prove \eqref{grad0-dS-bnd-1} by inducting on the dimension in the same way as Lemma \ref{lem:DZ-order}.

The conclusion is trivial in dimension $d=1$.
Suppose that the conclusion holds for dimension $(d-1)$, and split into cases where $\mu_d$ is bounded or unbounded.
If $|\mu_d|< X$ is bounded for some constant $X>1$, then by the continuity and strict positivity of the mapping $z=z(\mu)$, all components $z_j$ are bounded away from zero by some fixed amount $z_j>\delta>0$. Therefore,
\begin{equation} \label{mu-1-2z-bnd}
    \sup_{-X<\mu_d\le \cdots\le \mu_2<0} \left|\mu_j+\frac{1}{2z_j} \right| \le C(X).
\end{equation}
Otherwise, $\mu_d\le -X$, and we use the asymptotic estimate \eqref{zj-asymp}:
\begin{equation} \label{zj-asymp-1}
    z_j = \begin{cases}
        z^{(d-1)}_j(\mu') + O(\llbracket\mu_j\rrbracket_1 |\mu_d|^{-1}), \ j\le d-1, \\
        (2|\mu_d|)^{-1} + O(|\mu_d|^{-2}), \ j=d,
    \end{cases}
\end{equation}
where $\mu'=[\mu_1,\ldots,\mu_{d-1}]^T,$ the subscript $(d-1)$ indicates a lower-dimensional distribution, and the coefficient of the $O(\cdot)$ term depends on $d$ only. For $j\le d-1$, we have that
\[\left| \mu_j+\frac{1}{2z_j}\right|= \frac{|2\mu_jz_j+1|}{2 z_j} = \frac{|2\mu_j z^{(d-1)}_j + 1| + O(|\mu_d|^{-1})}{2z_j},\]
where in the numerator $\mu_j\cdot O(\llbracket \mu_j \rrbracket_1 |\mu_d|^{-1}) = O(|\mu_d|^{-1})$ because $|\mu_j\cdot \llbracket \mu_j \rrbracket_1| \le 1$ by the definition of the modulus \eqref{w-alpha}.
By \eqref{za-order}, $z_j \asymp \llbracket \mu_j \rrbracket_1$ with a coefficient only dependent on $d$. Since $|\mu_d|\ge |\mu_j|$ by assumption, we can control the error term by
\[\frac{O(|\mu_d|^{-1})}{z_j} = O\qty(\frac{|\mu_j|}{|\mu_d|}) = O(1).\]
Then, using the induction hypothesis in $(d-1)$ dimensions, we have that
\[|2\mu_j z_j^{(d-1)}+1|= O(z_j^{(d-1)}) =O(\llbracket \mu_j \rrbracket_1)\]
as the estimate \eqref{za-order} holds in all dimensions; hence, the first term is also bounded by a constant.
For $j=d$, we have that
\[ \frac{|2\mu_dz_d+1|}{2 z_d} = \frac{O(|\mu_d|^{-1})}{2z_d},\]
which is also uniformly bounded by the order condition $z_d\asymp |\mu_d|^{-1}$ \eqref{za-order}.
Summarizing the above, we get
\begin{equation} \label{mu-1-2z-ubnd}
    \sup_{\mu_d\le -X} \left| \mu_j + \frac{1}{2z_j} \right| < \infty
\end{equation}
as desired. Combining the results \eqref{mu-1-2z-bnd} and \eqref{mu-1-2z-ubnd} yields \eqref{grad0-dS-bnd-1}, which then proves Theorem \ref{thm:Sq-asymp}.
\end{proof}

Now that $\Delta S(q)$ is Lipschitz continuous and extends to the boundary of $\Delta_0^d$, a natural question is how this function behaves on the boundary. 
The domain of $S$ is a $(d-1)$-dimensional simplex, which has an intrinsic embedded structure---lower-dimensional simplices appear as faces of higher-dimensional ones. Therefore, a straightforward guess is that the function $\Delta S$ equals its lower-dimensional counterpart on the faces of $\Delta_0^d$. The rigorous statement is given in the following theorem.
\begin{theorem} \label{thm:dS-embed}
Denote by $\Delta S^{(d-1)}$ the residual term from \eqref{Sq-asymp} in $(d-1)$ dimensions. Then, for any $q'\in \Delta_0^{d-1}$,
\begin{equation} \label{dS-embed}
    \Delta S^{(d-1)}(q') = \Delta S \qty(q'+\frac{1}{d(d-1)}, -\frac1d) + \frac{1+\ln 2}{2}+\ln\frac{\Gamma(\tfrac{d}{2})}{\Gamma(\tfrac{d-1}{2})},
\end{equation}
where $\Delta S$ is defined by its limiting value when one argument equals $-\frac1d$. That is, $\Delta S^{(d-1)}$ is identical to $\Delta S$ on a $(d-2)$-dimensional face of $\Delta_0^d$, up to a constant shift.
\end{theorem}
\begin{proof}
Since $\Delta S$ is continuous up to the boundary of $\Delta_0^d$, we simply need to prove that the limit of $\Delta S$ as the $d$-th component of $q$ approaches $-\frac1d$ equals the LHS. 
We fix $q'$ on the relative interior of $\Delta_0^{d-1}$, let $\mu'$ be the unique $(d-1)$-dimensional parameter associated with $q'$ (guaranteed by Proposition \ref{prop:qj-surj}). Then, we send only $\mu_d\to-\infty$, holding $\mu'=[\mu_1,\ldots,\mu_{d-1}]^T$ fixed.

We quote the asymptotic expansion of $z_j$ when $\mu_d\to-\infty$ from \eqref{zj-asymp-1}, as well as the expansion of $Z$ from \eqref{Mdk-asymp-mult}:
\[Z(\mu) = \frac{\Gamma(\tfrac{d}{2})}{\Gamma(\tfrac{d-1}{2})} \frac{Z^{(d-1)}(\mu')}{\sqrt{|\mu_d|}} + O(\llbracket\mu\rrbracket_{\frac12}|\mu_d|^{-1}),\]
where the coefficient of the $O(\cdot)$ term depends on $d$ only. Then, we get the following limits:
\begin{align*}
    \lim_{\mu_d\to-\infty} z_j &= z_j^{(d-1)}(\mu'),\ j=1,2,\ldots,d-1, \\
    \lim_{\mu_d\to-\infty} |\mu_d| z_d &= \frac12, \\
    \lim_{\mu_d\to-\infty} \sqrt{|\mu_d|} Z &= \frac{\Gamma(\tfrac{d}{2})}{\Gamma(\tfrac{d-1}{2})} Z^{(d-1)}(\mu').
\end{align*}
Hence, the limit of the point $q(\mu)$ lies on a $(d-2)$-dimensional face of $\Delta_0^d$:
\begin{equation} \label{q-face-limit}
    \lim_{\mu_d\to-\infty} q(\mu)
=\lim_{\mu_d\to-\infty}\qty[z(\mu)-\frac1d]=\begin{bmatrix}
    z^{(d-1)}(\mu') -\frac1d \\ -\frac1d
\end{bmatrix} \\
= \begin{bmatrix}
    q' +\frac1{d(d-1)} \\ -\frac1d
\end{bmatrix}.
\end{equation}
Note that the limit point is indeed $q'=q^{(d-1)}(\mu')$ if we identify the $(d-2)$-dimensional face with $\Delta_0^{d-1}$.

Then, we study the limit of $\Delta S=S-\hat S$ as $\mu_d\to-\infty$:
\[\lim_{\mu_d\to-\infty} \Delta S(q(\mu)) =
    \lim_{\mu_d\to-\infty} \qty[\sum_{j=1}^d\mu_j z_j - \ln Z + \frac12 \ln(z_1\cdots z_d)].\]
Extracting the bounded terms corresponding to the components $\mu'$, we get
\[\lim_{\mu_d\to-\infty} \Delta S(q(\mu)) = \sum_{j=1}^{d-1} \mu_j z_j^{(d-1)} + \frac12\ln(z_1^{(d-1)}\cdots z_{d-1}^{(d-1)}) + \lim_{\mu_d\to-\infty} \qty(\mu_d z_d + \ln\frac{\sqrt{z_d}}{Z}).\]
We use the asymptotic relations $z_d\sim (2|\mu_d|)^{-1}$ and $Z \sim |\mu_d|^{-\frac12} \frac{\Gamma(\tfrac{d}{2})}{\Gamma(\tfrac{d-1}{2})} Z^{(d-1)}$ on the final term. Eventually, the logarithmic singularities cancel out, and what remains is the desired constant value:
\begin{equation} \label{dS-face-limit}
\begin{aligned}
    \lim_{\mu_d\to-\infty} \Delta S(q(\mu)) &= \sum_{j=1}^{d-1} \mu_j z_j^{(d-1)} + \frac12\ln(z_1^{(d-1)}\cdots z_{d-1}^{(d-1)}) \\
    &\quad -\frac12 -\ln Z^{(d-1)} -\ln\frac{\Gamma(\tfrac{d}{2})}{\Gamma(\tfrac{d-1}{2})} -\frac12\ln 2 \\
    &=\Delta S^{(d-1)}(q^{(d-1)}(\mu')) - \frac{1+\ln 2}{2}-\ln\frac{\Gamma(\tfrac{d}{2})}{\Gamma(\tfrac{d-1}{2})}.
\end{aligned}
\end{equation}
We have substituted the expression of $\Delta S^{(d-1)}$ in $(d-1)$ dimensions. By the continuity of $\Delta S$ up to the boundary, the value of $\Delta S$ at the limit point \eqref{q-face-limit} is exactly the limit of the function values \eqref{dS-face-limit}. Reorganizing the terms, we obtain the exact formula \eqref{dS-embed} at $q'=q^{(d-1)}(\mu')$:
\[\Delta S^{(d-1)}(q') = \Delta S\qty(q'+\frac{1}{d(d-1)}, -\frac1d)+ \frac{1+\ln 2}{2}+\ln\frac{\Gamma(\tfrac{d}{2})}{\Gamma(\tfrac{d-1}{2})}.\]
Finally, due to the continuity of $\Delta S$ and $\Delta S^{(d-1)}$, the fact that they satisfy \eqref{dS-embed} on the relative interior of a $(d-2)$-dimensional face indicates that the same relation holds for the closure of the face, as desired.
\end{proof}

\begin{remark}
A natural corollary of Theorem \ref{thm:dS-embed} is that the profile of $\Delta S$ on each $(s-1)$-dimensional face of $\Delta_0^d$ is identical to its counterpart $\Delta S^{(s)}$ up to a constant shift. More specifically, if $q'\in S_0^{s}$ is identified with the point $[(q')^T+\frac1s-\frac1d, -\frac1d,\cdots,-\frac1d]^T$ on a $(s-1)$-dimensional face of $\Delta_0^d$, then by connecting the equalities \eqref{dS-embed} from dimensions $d$ to $s+1$, we get
\begin{equation} \label{dS-embed-s}
    \Delta S^{(s)}(q') = \Delta S\qty(q'+\frac1s-\frac1d, -\frac1d,\cdots,-\frac1d) + (d-s)\frac{1+\ln 2}{2} + \ln \frac{\Gamma(\frac{d}{2})}{\Gamma(\frac{s}{2})}.
\end{equation}
In particular, when $d=1$, the Bingham closure is trivial with $q^{(1)}\equiv 0,\ S^{(1)}\equiv 0$, so by taking $s=1$ in \eqref{dS-embed}, we get the value of the $d$-dimensional $\Delta S$ at the vertices of $\Delta_0^d$:
\begin{equation}
    \Delta S\qty(1-\frac1d,-\frac1d,\cdots,-\frac1d)=-(d-1)\frac{1+\ln 2}{2} - \ln \frac{\Gamma(\frac{d}{2})}{\sqrt{\pi}}.
\end{equation}
Letting $d=3$, we find that $\Delta S^{(3)}(\frac23,-\frac13,-\frac13)=-(1+\ln 2) - \ln \frac12=-1,$ which is exactly the conclusion we obtained in \cite{shi_molecular_2026}.
\end{remark}

Theorem \ref{thm:dS-embed} fortifies the idea of dimensional hierarchy in Bingham distributions: lower-dimensional Bingham distributions appear as the limit of their higher-dimensional counterparts when one or more of the eigenvalues $\mu_j$ go to $-\infty$, or equivalently, some second-order moments $z_j$ shrink to zero.
This fact has already been embodied in the asymptotic expansions (equations \eqref{DMdk-asymp-mult}, \eqref{za-asymp}, etc.), where the function $Z$ and moments $z_\alpha$ from lower dimensions appear in the leading terms of the asymptotic series.


\section{Conclusion}

In this paper, we have developed a uniform asymptotic framework for the singular regimes of the Bingham distribution in arbitrary dimensions, allowing multiple eigenvalues of $B$ to diverge without restrictions on their relative scales. 
Starting from an inverse-Laplace representation of the normalizing constant, we deformed the contour to obtain exponentially convergent formulas for $Z$ and its derivatives. These formulas yield asymptotic expansions with error estimates uniform in all eigenvalue configurations, regardless of how the unbounded eigenvalues diverge. Notably, the coefficients of these expressions come from the lower-dimensional Bingham quantities, reflecting the dimensional-hierarchical structure. In addition, the asymptotic series has an absolutely convergent limit when the number of bounded eigenvalues is even.

The expansions imply precise asymptotics for Bingham moments such as $z_j\sim\frac{1}{2|\mu_j|}$ ($\mu_j\to-\infty$), and provide useful initial approximations for numerical Bingham closure. They also lead to the entropy decomposition
\[ S(q)=-\frac12\ln(z_1\cdots z_d)+\Delta S(q), \]
where $\Delta S$ is uniformly Lipschitz on the moment simplex and extends continuously to its boundary.
Moreover, the residual agrees, up to an explicit additive constant, with the residual for the corresponding lower-dimensional Bingham distribution. Thus, dimensional hierarchy is reflected simultaneously in $Z$, the moments, and the entropy.

The conclusions of Theorems \ref{thm:Sq-asymp} and \ref{thm:dS-embed} are also directly applicable to the study of anisotropic materials.
Under the (non-diagonalized) Bingham distribution $\mathrm{Bing}(B)$ with $\tr B=0$ and its corresponding $Q$-tensor $Q=\BbbE(mm^T-\frac1d I)$, the decomposition is translated into
\[S(Q) = B:Q-\ln Z(B) = -\frac12\ln\det(Q+\frac1d I) + \Delta S(Q),\]
where $S(Q)=S(\lambda(Q))$ also has Lipschitz continuity by Weyl's inequality on the eigenvalues \cite{golub_matrix_2013}: $\|\lambda(A)-\lambda(B)\|\le c\|A-B\|.$ 
In the study of liquid crystals, this formulation enables efficient evaluation of the mean-field energy through various approximation techniques \cite{shi_molecular_2026}, and accelerates the computation of molecular-based $Q$-tensor dynamics \cite{han_microscopic_2015}.

Several directions remain open. Higher-order moment expansions may improve closure algorithms, while sharper analysis of derivatives of the entropy could establish higher boundary regularity of $\Delta S$, in particular $C^2$ regularity.

\bibliographystyle{siam}
\bibliography{Bingham.bib}

\appendix

\section{Derivation of the integral representation} \label{app:proof-of-lt}

We prove Theorem \ref{thm:Z-lt}.
\begin{proof}[Proof of Theorem \ref{thm:Z-lt}]
Using the Gaussian integral and analytical continuation, we have that
\begin{equation}\label{comp-gauss-int}
    \int_{-\infty}^\infty \e^{-ax^2}\d x = \left(\frac{\pi}{a} \right)^{\frac12},\ \Re a>0,
\end{equation}
where the complex square root follows the same convention as stated in the proposition. 
For any $\lambda\in\BbbC$ with $\Re\lambda>\max\{\mu_i\}$, we have by \eqref{comp-gauss-int} that
\begin{equation} \label{J-lambda}
    J(\lambda)=\int_{\BbbR^d} \e^{\sum_j (\mu_j-\lambda) x_j^2} \d x_j = \pi^{\frac{d}{2}}\prod_{j=1}^d (\lambda-\mu_j)^{-\frac12}.
\end{equation}
It can also be expressed by radial integration:
\begin{align*}
    J(\lambda) &= \int_0^\infty \d r \int_{\pt B_r} \e^{\sum_j (\mu_j-\lambda) x_j^2} \d S(x) \\
    &= \int_0^\infty r^{d-1} \e^{-\lambda r^2}\cdot \omega_d Z(r^2\mu)\d r.
\end{align*}
Using a change of variable $r^2\to t$, we find that $J(\lambda)$ is in fact a Laplace transform of the function $f(t) = \frac{\omega_d}{2} t^{\frac{d}{2}-1} Z(t\mu)$:
\begin{equation}
    J(\lambda) = \frac{\omega_d}{2} \int_0^\infty \e^{-\lambda t} t^{\frac{d}{2}-1} Z(t\mu)\d t.
\end{equation}
Using the inverse Laplace transform \cite{abramowitz_handbook_2013} on $J(\lambda)$ and evaluating $f(t)$ at $t=1$, we immediately get \eqref{Zd-is-inv-lt}.
\end{proof}

\section{Differential equations of the normalizing constant}

In the appendix, we discuss some differential equations satisfied by the Bingham normalizing constant $Z(\mu)$ and its derivatives \cite{sei_calculating_2015}. This result immediately yields the moment identity \eqref{zj-zk-zjk-id}.
\begin{proposition} \label{prop:Zd-pfaff}
Suppose that $\mu\in\BbbR^d$ and that $Z$ is defined by \eqref{Z-mu}. Then,
\begin{align}
    \sum_j \pt_j Z &= Z, \label{sum-Zj-is-Z} \\
    (\pt_j-\pt_k)Z &= 2(\mu_j-\mu_k)\pt_{jk} Z. \label{zj-zk-zjk}
\end{align}
\end{proposition}
\begin{proof}
We quote the integral representation \eqref{Zd-hankel} on a Hankel-type contour $C$. 

Since the function $F(z)=\e^z\prod_j (z-\mu_j)^{-\frac12}$ converges to 0 on both ends of the contour, we have that
\[0 = \oint_C F'(z)\d z
=\oint_C F(z) \d z - \sum_{j=1}^d \frac12 \oint_C \frac{F(z)}{z-\mu_j} \d z.\]
Up to dividing by the constant $\frac{\Gamma(\frac{d}{2})}{2\pi\ii}$, the first term equals $Z(\mu)$, while the second term above equals the derivatives $\pt_j Z$ by \eqref{DZd-hankel}, so \eqref{sum-Zj-is-Z} is correct.

As for \eqref{zj-zk-zjk} (with $j\neq k$, trivial otherwise), we utilize \eqref{DZd-hankel} again to get that
\begin{align*}
    \text{LHS} &= \frac{\Gamma(\frac{d}{2})}{2\pi\ii} \cdot \frac12 \oint_C F(z)\qty(\frac{1}{z-\mu_j} - \frac{1}{z-\mu_k}) \d z \\
    &=\frac{\Gamma(\frac{d}{2})}{2\pi\ii} \cdot \frac12 \oint_C F(z)\frac{\mu_j-\mu_k}{(z-\mu_j)(z-\mu_k)} \d z=\text{RHS},
\end{align*}
where in the last step we extract the constant factor $(\mu_j-\mu_k)$ from the integral to find the remaining terms to be identical to $2\pt_{jk} Z$.
\end{proof}


\end{document}